\documentclass[a4paper,reqno,11pt]{amsart}
\usepackage{}

\usepackage{graphicx}
\usepackage{amssymb}
\usepackage{amstext}
\usepackage{amsmath}
\usepackage{amscd}
\usepackage{amsthm}
\usepackage{amsfonts}
\usepackage{enumitem}
\usepackage{hyperref}
\usepackage{cleveref}
\usepackage{color}
\usepackage{latexsym}
\usepackage[all,2cell,color]{xy}
\usepackage{comment}

\usepackage{tikz}
\usetikzlibrary{matrix,arrows,calc,snakes,patterns,decorations.markings}
\tikzset{black/.style={circle,fill=black,inner sep=3pt,outer sep=3pt},
         white/.style={circle,fill=white,draw=black,inner sep=3pt,outer sep=3pt},
}

\usepackage{booktabs}
\usepackage{longtable}
\usepackage{caption}
\usepackage{multirow}
\usepackage{epstopdf}
\newcommand{\opname}[1]{\operatorname{\mathsf{#1}}}

\renewcommand{\mod}{\opname{mod}\nolimits}
\newcommand{\brick}{\opname{brick}}
\newcommand{\sbrick}{\opname{sbrick}}
\newcommand{\mbrick}{\opname{mbrick}}

\numberwithin{equation}{section}
\newtheorem{theorem}{Theorem}[section]

\newtheorem{corollary}[theorem]{Corollary}
\newtheorem{lemma}[theorem]{Lemma}
\newtheorem{proposition}[theorem]{Proposition}
\newtheorem{definition-theorem}[theorem]{Definition-Theorem}
\newtheorem{definition-proposition}[theorem]{Definition-Proposition}

\theoremstyle{definition}
\newtheorem{definition}[theorem]{Definition}

\newtheorem{remark}[theorem]{Remark}
\newtheorem{example}[theorem]{Example}

\newcommand{\Hom}{\operatorname{Hom}\nolimits}

\newcommand{\End}{\operatorname{End}\nolimits}

\newcommand{\RHom}{\mathbf{R}\strut\kern-.2em\operatorname{Hom}\nolimits}

\DeclareMathOperator{\moduleCategory}{\mathsf{mod}} \renewcommand{\mod}{\moduleCategory}

\begin{document}

\title{Left Schur subcategories in recollements}
%\date{\today}

\author[Yingying Zhang]{Yingying Zhang*}
\address{Yingying Zhang,
Department of Mathematics, Huzhou Normal University, Huzhou 313000, Zhejiang Province, P.R.China}
\email{yyzhang@huznu.edu.cn}

\author[Dajun Song]{Dajun Song}
\address{Dajun Song,
Department of Mathematics, Huzhou Normal University, Huzhou 313000, Zhejiang Province, P.R.China}
\email{208063861@qq.com}
\thanks{MSC2020: 16G10, 18A40, 18E10}
\thanks{Keywords: monobrick, left Schur subcategory, recollement, length abelian category}
\thanks{*Corresponding author}

\begin{abstract}
Recently left Schur subcategories in a length abelian category were introduced by Enomoto, which unify torsion-free classes and wide subcategories. In this paper, we give a construction of left Schur subcategories in the recollements of length abelian categories. Moreover, we show that the construction restricts to wide subcategories and torsion-free classes. As a consequence, we obtain an explicit construction of cofinally closed monobricks in recollements. Finally, we apply the results to a special case of the MacPherson-Vilonen construction of abelian recollements and triangular matrix algebras.
\end{abstract}

\maketitle

\section{Introduction}

In representation theory of finite-dimensional algebras, (semi)bricks are a generalization of (semi)simple modules, and they have long been studied \cite{Rin}. The bijection between semibricks and wide subcategories was mentioned in \cite[1.2]{Rin} under the name ``simplification". These wide subcategories were further studied by Hovey \cite{H} and they provide a significant interface between representation theory and combinatorics \cite{IT}. Adachi, Iyama and Reiten introduced support $\tau$-tilting modules and established their bijective correspondence with functorially finite torsion classes \cite{AIR}. Subsequently, Demonet, Iyama and Jasso investigated the relation between bricks and $\tau$-rigid modules, giving rise to the brick-$\tau$-rigid correspondence \cite{DIJ}. Asai gave an extension of their result and showed that support $\tau$-tilting modules correspond bijectively to left finite semibricks \cite{As}. Recently, as a generalization of semibricks, Enomoto introduced \emph{monobricks} in a length abelian category \cite{E}. He established a bijection between monobricks and \emph{left Schur subcategories}-a novel class that uniformly encompasses torsion-free classes and wide subcategories. This unified framework serves as the conceptual starting point of our work.

A recollement of abelian categories is a short exact sequence of abelian categories where both the inclusion and the quotient functors admit left and right adjoints. They first appeared in the construction of the category of perverse sheaves on a singular space by Beilinson, Bernstein and Deligne \cite{BBD}, arising from recollements of triangulated categories. A fundamental example of a recollement of abelian categories is due to MacPherson and Vilonen \cite{MV}. Gluing techniques with respect to a recollement have been intensively studied in \cite{BBD} for t-structures and simple objects. Torsion pairs in recollements of abelian categories have been investigated by Ma and Huang in \cite{MH}. In \cite{Zh, Z}, the first author studied gluing techniques for semibricks and support $\tau$-tilting modules. Recently, Yang and Qin explored monobricks and ICE-closed subcategories over recollements \cite{YQ}. In this paper, we unify certain aspects of the works from [12, 16, 17, 18] within recollements of length abelian categories via the notion of left Schur subcategories introduced by Enomoto. Our main result is as follows.
\begin{theorem}{\rm(Theorem 3.2)}
Let $R(\mathcal {Y}, \mathcal {X}, \mathcal {Z})$ be a recollement of $\mathcal {X}$ relative to $\mathcal {Y}$ and $\mathcal {Z}$ as follows:
\begin{equation}\label{recollement}
\xymatrix@!C=2pc{
\mathcal {Y}\; \ar@{>->}[rr]|{i_{*}} && \mathcal {X} \ar@<-4.0mm>@{->>}[ll]_{i^{*}} \ar@{->>}[rr]|{j^{*}} \ar@{->>}@<4.0mm>[ll]^{i^{!}}&& \mathcal {Z} \ar@{>->}@<-4.0mm>[ll]_{j_{!}} \ar@{>->}@<4.0mm>[ll]^{j_{*}}
}.
\end{equation}
Assume that $\mathcal {E}_{\mathcal {Y}}$ and $\mathcal {E}_{\mathcal {Z}}$ are subcategories of $\mathcal {Y}$ and $\mathcal {Z}$ respectively, and let $\mathcal {E}_{\mathcal {X}}=\{X\in \mathcal {X}|j^{*}X\in\mathcal {E}_{\mathcal {Z}}, i^{!}X\in\mathcal {E}_{\mathcal {Y}}\}$. If $i^{!}$ is exact, then the following are equivalent:
\begin{itemize}
\item[\rm(1)] $\mathcal {E}_{\mathcal {Y}}$ and $\mathcal {E}_{\mathcal {Z}}$ are left Schur.
\item[\rm(2)] $\mathcal {E}_{\mathcal {X}}$ is left Schur.
\end{itemize}
\end{theorem}
Beyond its theoretical interest, this theorem provides a systematic recipe for constructing new left Schur subcategories in $\mathcal {X}$ from $\mathcal {Y}$ and $\mathcal {Z}$. Moreover, the construction restricts to wide subcategories and torsion-free classes. This is particularly useful for idempotent recollements associated with the triangular matrix algebras, which satisfy the exactness of $i^{!}$ automatically. We demonstrate this in Section 4 with explicit examples, visualized in Auslander$-$Reiten quivers.

The rest of this paper is organized as follows. In Section $\ref{sec:2}$, we recall some preliminary results on monobricks, left Schur subcategories and recollements needed for the later sections. In Section $\ref{sec:3}$, we mainly investigate the relation between left Schur subcategories in the middle category $\mathcal {X}$ of (\ref{recollement}) and those in $\mathcal {Y}$ and $\mathcal {Z}$. In Section $\ref{sec:4}$, we apply the results to a special case of the MacPherson-Vilonen construction of abelian recollements and triangular matrix algebras. Finally, we give examples to illustrate Theorem 1.1.

\section{Preliminaries}\label{sec:2}
Throughout this section, let $\mathcal{A}$ be a length abelian category. $\mathcal{E}$ will always denote a subcategory of $\mathcal{A}$. For a set $S$, we denote by $\# S$ its cardinality.
\subsection{Monobricks and left Schur subcategories}

First we recall the definition of semibricks and monobricks.
\begin{definition}

(1) An object $S\in\mathcal{A}$ is called a \emph{brick} if $\End_{A}(S)$ is a division ring, that is, such that all non-trivial endomorphisms are invertible. We write $\brick \mathcal{A}$ for the set of isoclasses of bricks in $\mathcal{A}$.

(2) \cite[Definition 2.1]{As} A subset $\mathcal{S}\subset \brick \mathcal{A}$ is called a \emph{semibrick} if $\Hom_{\mathcal{A}}(S_{1}, S_{2})=0$ for any $S_{1}\neq S_{2}\in\mathcal{S}$. We write $\sbrick \mathcal{A}$ for the set of semibricks in $\mathcal{A}$.

(3) \cite[Definition 2.1]{E} A subset $\mathcal{M}\subset \brick \mathcal{A}$ is called a \emph{monobrick} if every $f\in\Hom_{\mathcal{A}}(S_{1}, S_{2})$ is either zero or an injection for any $S_{1}, S_{2}\in\mathcal{M}$. We write $\mbrick \mathcal{A}$ for the set of monobricks in $\mathcal{A}$.

(4) \cite[Corollary 3.9]{E} A monobrick $\mathcal{M}\in \mbrick \mathcal{A}$ is called \emph{cofinally closed} if it satisfies the following condition: if there is an injection $N\hookrightarrow M$ for a brick $N\notin \mathcal{M}$ and $M\in \mathcal{M}$, then there is a non-zero non-injection $N\rightarrow M'$ for some $M'\in \mathcal{M}$. We write $\mbrick_{\text{c.c.}} \mathcal{A}$ for the set of cofinally closed monobricks in $\mathcal{A}$.
\end{definition}

Note that the assumption that a monobrick $\mathcal{M}$ consists of bricks is automatically satisfied for a length abelian category. Every semibrick is obviously a monobrick. Next we recall the definition of left Schur subcategories. Before this, we need the following notions.
\begin{definition}
(1) $\mathcal{E}$ is \emph{closed under extension} or \emph{extension-closed} in $\mathcal{A}$ if it satisfies: for any short exact sequence $0\longrightarrow X\longrightarrow Y\longrightarrow Z\longrightarrow 0$ in $\mathcal{A}$, if $X\in \mathcal{E}$ and $Z\in \mathcal{E}$, then $Y\in \mathcal{E}$.

(2) Suppose that $\mathcal{E}$ is extension-closed in $\mathcal{A}$. Then a non-zero object $M\in \mathcal{E}$ is a \emph{simple object} in $\mathcal{E}$ if there is no exact sequence of the form $0\longrightarrow L\longrightarrow M\longrightarrow N\longrightarrow 0$ in $\mathcal{A}$ satisfying $L, M, N\in \mathcal{E}$ and $L, N\neq 0$. We denote by $\opname{sim} \mathcal{E}$ the set of isoclasses of simple objects in $\mathcal{E}$.

(3) \cite[Definition 2.3]{E} Let $\mathcal{C}$ be a collection of objects in $\mathcal{A}$. Then a non-zero object $M\in \mathcal{A}$ is \emph{left Schurian} for $\mathcal{C}$ if every morphism $M\longrightarrow C$ with $C\in \mathcal{C}$ is either zero or an injection in $\mathcal{A}$.
\end{definition}
If $\mathcal{A}$ is an abelian category, then $\opname{sim} \mathcal{A}$ is nothing but the set of simple objects of $\mathcal{A}$. Thus $\opname{sim} \mathcal{E}$ is an analogue of simple objects inside $\mathcal{E}$.

\begin{definition}\cite[Definition 2.5]{E}
A subcategory $\mathcal{E}$ of $\mathcal{A}$ is \emph{left Schur} if it is extension-closed in $\mathcal{A}$ and if every simple object in $\mathcal{E}$ is left Schurian for $\mathcal{E}$. We denote by $\opname{Schur_L} \mathcal{A}$ the set of left Schur subcategories of $\mathcal{A}$.
\end{definition}

Roughly speaking, left Schur subcategories are extension-closed subcategories of $\mathcal{A}$ such that the ``one-side Schur's lemma" holds. It is shown in \cite[Proposition 2.8]{E} that the class of left Schur subcategories contains any subcategory which is closed under extensions, kernels and images, and thus unifies torsion-free classes and wide subcategories. Let us recall the definition of these subcategories.
\begin{definition}
(1) $\mathcal{E}$ is \emph{closed under kernels (resp. cokernels)} if for every morphism $f: X\rightarrow Y$ in $\mathcal{E}$, we have that $\opname{Ker}f$ (resp. ${\opname{Coker}f}$) belongs to $\mathcal{E}$.

(2) $\mathcal{E}$ is a \emph{torsion-free class} in $\mathcal{A}$ if it is closed under extensions and subobjects in $\mathcal{A}$. We denote by $\opname{torf} \mathcal{A}$ the set of torsion-free classes in $\mathcal{A}$.

(3) $\mathcal{E}$ is a \emph{wide subcategory} in $\mathcal{A}$ if it is closed under extensions, kernels and cokernels. We denote by $\opname{wide} \mathcal{A}$ the set of wide subcategories of $\mathcal{A}$.
\end{definition}

The following result gives a bijection between $\opname{Schur_{L}}\mathcal{A}$ and $\mbrick \mathcal{A}$. Let $\mathcal {C}$ be a collection of objects in $\mathcal {A}$. $\opname{Filt}\mathcal {C}$ denotes the subcategory of $\mathcal {A}$ consisting of objects $X$ such that there is a $\mathcal {C}$-filtration $0=X_0\subset X_1\subset \cdots \subset X_n=X$ of subobjects of $X$ such that $X_i/X_{i-1}\in \mathcal {C}$ for each $i$.

\begin{theorem}{\rm\cite[Theorem 2.11]{E}}\label{bijection}
There is a bijection
\begin{center}
$\opname{Schur_{L}}\mathcal{A} \longleftrightarrow \mbrick \mathcal{A}$
\end{center}
given by $\opname{Schur_{L}}\mathcal{A} \ni \mathcal{E}  \mapsto \opname{sim}\mathcal{E} \in \mbrick \mathcal{A}$ and $\mbrick \mathcal{A} \ni \mathcal{M} \mapsto \opname{Filt}\mathcal{M} \in\opname{Schur_{L}}\mathcal{A}$, which restricts to the following bijections:
\begin{center}
$\opname{wide}\mathcal{A} \longleftrightarrow \sbrick \mathcal{A}$ and $\opname{torf}\mathcal{A} \longleftrightarrow \mbrick_{\rm{c.c.}} \mathcal{A}$.
\end{center}
\end{theorem}

\subsection{Recollements}

In this subsection, we recall the notion of recollements of abelian categories, that is central to our results \cite{BBD, FP}.
\begin{definition} Let $\mathcal {X}, \mathcal {Y}, \mathcal {Z} $ be abelian categories. Then a recollement of $\mathcal {X}$ relative to $\mathcal {Y}$ and $\mathcal {Z}$ is diagrammatically expressed by
$$\xymatrix@!C=2pc{
\mathcal {Y}\; \ar@{>->}[rr]|{i_{*}} && \mathcal {X} \ar@<-4.0mm>@{->>}[ll]_{i^{*}} \ar@{->>}[rr]|{j^{*}} \ar@{->>}@<4.0mm>[ll]^{i^{!}}&& \mathcal {Z} \ar@{>->}@<-4.0mm>[ll]_{j_{!}} \ar@{>->}@<4.0mm>[ll]^{j_{*}}
}$$
in which the arrows satisfy the following three conditions:
\begin{itemize}
\item[(1)] ($i^{*}, i_{*}$), ($i_{*}, i^{!}$), ($j_{!}, j^{*}$) and ($j^{*}, j_{*}$) are adjoint pairs;
\item[(2)] $i_{*}, j_{!}$ and $j_{*}$ are fully faithful functors;
\item[(3)] ${\rm Im}i_{*}={\rm Ker}j^{*}$.
\end{itemize}
\end{definition}

\begin{remark}\label{six-functor}
(1) From Definition 2.6(1), it follows that $i_{*}$ and $j^{*}$ are both right adjoint functors and left adjoint functors, therefore they are exact functors of abelian categories.

(2) By the definition of recollements, we have $i^{*}i_{*}\cong id_\mathcal {Y}, i^{!}i_{*}\cong id_\mathcal {Y}, j^{*}j_{!}\cong id_\mathcal {Z}$ and $j^{*}j_{*}\cong id_\mathcal {Z}$. Also $i^{*}j_{!}=0, i^{!}j_{*}=0$.

(3) Associated to a recollement there is a seventh functor $j_{!*}:= \opname{Im} (j_!\rightarrow j_*): \mathcal {Z}\rightarrow \mathcal {X}$ called the intermediate extension functor. This functor plays an important role in gluing simple objects.

(4) From now on, we assume that $\mathcal {X}, \mathcal {Y}, \mathcal {Z} $ are all length abelian categories. For example, finitely generated module categories over finite-dimensional algebra are length abelian categories. We denote by $R(\mathcal {Y}, \mathcal {X}, \mathcal {Z})$ a recollement of $\mathcal {X}$ relative to $\mathcal {Y}$ and $\mathcal {Z}$ as above.
\end{remark}

The following proposition summarizes results in \cite{BBD, FP, YQ, Zh}

\begin{proposition}\label{intermediate}
\rm{(1)} \cite[Propositions 4.4 and 8.8]{FP} If $i^!$ is exact, then $i^*j_*=0$ and $j_{!*}\cong j_*$.

(2) \cite[Section 1.4]{BBD} Every simple object in $\mathcal {X}$ is either of the form $i_*S$ for some simple object in $\mathcal {Y}$ or of the form $j_{!*}S$ for some simple object in $\mathcal {Z}$.

(3) \cite[Theorem 2.1]{Zh} If $\mathcal {S}_\mathcal {Y}\in \sbrick \mathcal {Y}$ and $\mathcal {S}_\mathcal {Z}\in \sbrick \mathcal {Z}$, then $i_*(\mathcal {S}_\mathcal {Y})\sqcup j_{!*}(\mathcal {S}_\mathcal {Z})\in \sbrick \mathcal {X}$.

(4) \cite[Theorem 4.3]{YQ} If $\mathcal {M}_\mathcal {Y}\in \mbrick \mathcal {Y}$ and $\mathcal {M}_\mathcal {Z}\in \mbrick \mathcal {Z}$, then $i_*(\mathcal {M}_\mathcal {Y})\sqcup j_{!*}(\mathcal {M}_\mathcal {Z})\in \mbrick \mathcal {X}$.
\end{proposition}
For later use, we recall the property of recollements of abelian categories.
\begin{lemma}\cite[Lemma 3.1(4)]{FZ}\label{exact-functor}
The following are equivalent:
\begin{itemize}
\item[\rm(1)] $i^!$ is exact.
\item[\rm(2)] $i^!$ and $j_*$ are exact.
\item[\rm(3)] ${\rm Im}j_{*}={\rm Ker}i^{!}$ and for each object $X\in \mathcal{X}$ there is an exact sequence $0\rightarrow i_{*}i^{!}X\rightarrow X\rightarrow j_{*}j^{*}X\rightarrow 0$.
\end{itemize}
\end{lemma}

\section{Left Schur subcategories in recollements}\label{sec:3}

In this section, we give the construction of left Schur subcategories in recollements of length abelian categories. The following useful result may be well known to experts. For the convenience of the reader, we give a proof here.
\begin{proposition}\label{filtration}
Let $\mathcal {A}$ be an abelian category and $0\rightarrow X\rightarrow Y\buildrel {\pi} \over\rightarrow Z\rightarrow 0$ a short exact sequence in $\mathcal {A}$. If $X\in \opname{Filt}\mathcal {C}$ for some $\mathcal {C}\subset \mathcal {A}$ and $Z\in \opname{Filt}\mathcal {\mathcal {D}}$ for some $\mathcal {D}\subset \mathcal {A}$, then $Y\in \opname{Filt}(\mathcal {C}\cup  \mathcal {D})$.
\end{proposition}

\begin{proof}
Since $X\in \opname{Filt}\mathcal {C}$, we have a $\mathcal {C}$-filtration $0=X_0\subset X_1\subset \cdots \subset X_m=X$ with $X_i/X_{i-1}\in \mathcal {C}$ for each $i$. Since $Z\in \opname{Filt}\mathcal {D}$, we have a $\mathcal {D}$-filtration $0=Z_0\subset Z_1\subset \cdots \subset Z_n=Z$ with $Z_j/Z_{j-1}\in \mathcal {D}$ for each $j$. We can define $Y_j=\pi^{-1}(Z_j)$ for each $j$ since $\pi : Y\rightarrow Z$ is surjective. Then we have a chain $X\cong Y_0\subset Y_1\subset \cdots \subset Y_n=Y$ and the fact that $\pi$ is surjective implies that $Y_{j}/Y_{j-1}\cong Z_{j}/Z_{j-1}\in \mathcal {D}$ for each $j$. Now we can construct a chain $0=X_0\subset X_1\subset \cdots \subset X_m=X\cong Y_0\subset Y_1\subset \cdots \subset Y_n=Y$, where $X_i/X_{i-1}\in \mathcal {C}$ for each $i$ and $Y_{j}/Y_{j-1}\in \mathcal {D}$ for each $j$. This chain is exactly a $\mathcal {C}\cup \mathcal {D}$-filtration of $Y$ and therefore $Y\in \opname{Filt}(\mathcal {C}\cup  \mathcal {D})$.
\end{proof}

Now we are ready to prove the following main result in this paper.
\begin{theorem}\label{main theorem}
Let $R(\mathcal {Y}, \mathcal {X}, \mathcal {Z})$ be a recollement of length abelian categories
$$\xymatrix@!C=2pc{
\mathcal {Y}\; \ar@{>->}[rr]|{i_{*}} && \mathcal {X} \ar@<-4.0mm>@{->>}[ll]_{i^{*}} \ar@{->>}[rr]|{j^{*}} \ar@{->>}@<4.0mm>[ll]^{i^{!}}&& \mathcal {Z} \ar@{>->}@<-4.0mm>[ll]_{j_{!}} \ar@{>->}@<4.0mm>[ll]^{j_{*}}
}.$$ Assume that $\mathcal {E}_{\mathcal {Y}}$ and $\mathcal {E}_{\mathcal {Z}}$ are subcategories of $\mathcal {Y}$ and $\mathcal {Z}$ respectively, and let $\mathcal {E}_{\mathcal {X}}=\{X\in \mathcal {X}|j^{*}X\in\mathcal {E}_{\mathcal {Z}}, i^{!}X\in\mathcal {E}_{\mathcal {Y}}\}$. If $i^{!}$ is exact, then the following are equivalent:
\begin{itemize}
\item[\rm(1)] $\mathcal {E}_{\mathcal {Y}}$ and $\mathcal {E}_{\mathcal {Z}}$ are left Schur.
\item[\rm(2)] $\mathcal {E}_{\mathcal {X}}$ is left Schur.
\end{itemize}
\end{theorem}

\begin{proof}
$(1)\Rightarrow (2)$: Assume that $\mathcal {E}_{\mathcal {Y}}$ and $\mathcal {E}_{\mathcal {Z}}$ are left Schur. By Theorem \ref{bijection}, $\mathcal {M}_{\mathcal {Y}}:=\opname{sim} \mathcal {E}_{\mathcal {Y}}\in \mbrick \mathcal {Y}$ and $\mathcal {M}_{\mathcal {Z}}:=\opname{sim} \mathcal {E}_{\mathcal {Z}}\in \mbrick \mathcal {Z}$. It follows that $\mathcal {E}_{\mathcal {Y}}=\opname{Filt} \mathcal {M}_{\mathcal {Y}}$ and $\mathcal {E}_{\mathcal {Z}}=\opname{Filt} \mathcal {M}_{\mathcal {Z}}$. By Proposition \ref{intermediate}, we have $\mathcal {M}_{\mathcal {X}}:=i_*(\mathcal {M}_{\mathcal {Y}})\sqcup j_*(\mathcal {M}_{\mathcal {Z}})\in \mbrick \mathcal {X}$. We only have to prove that $\mathcal {E}_{\mathcal {X}}=\opname{Filt} \mathcal {M}_{\mathcal {X}}$. Thanks to Theorem \ref{bijection}, this fact would imply that $\mathcal {E}_{\mathcal {X}}$ is a left Schur subcategory.

Let $X\in \opname{Filt} \mathcal {M}_{\mathcal {X}}$. Since $i^!$ and $j^*$ are exact, by Remark \ref{six-functor} we have $i^!X\in i^!(\opname{Filt} \mathcal {M}_{\mathcal {X}})\subseteq \opname{Filt} i^!(\mathcal {M}_{\mathcal {X}})=\opname{Filt} \mathcal {M}_{\mathcal {Y}}=\mathcal {E}_{\mathcal {Y}}$, and $j^*X\in j^*(\opname{Filt} \mathcal {M}_{\mathcal {X}})\subseteq \opname{Filt} j^*(\mathcal {M}_{\mathcal {X}})=\opname{Filt} \mathcal {M}_{\mathcal {Z}}=\mathcal {E}_{\mathcal {Z}}$. It follows that $X\in \mathcal {E}_{\mathcal {X}}$. Thus $\opname{Filt} \mathcal {M}_{\mathcal {X}}\subset \mathcal {E}_{\mathcal {X}}$.

Conversely, suppose that $X\in \mathcal {E}_{\mathcal {X}}$. Since $i^!$ is exact, by Lemma \ref{exact-functor}, $j_*$ is also exact and we have an exact sequence $0\rightarrow i_{*}i^{!}X\rightarrow X\rightarrow j_{*}j^{*}X\rightarrow 0$. By the definition of $\mathcal {E}_{\mathcal {X}}$, we have $i_{*}i^{!}X\in i_*(\mathcal {E}_{\mathcal {Y}})=i_*(\opname{Filt} \mathcal {M}_{\mathcal {Y}})\subseteq \opname{Filt} i_*(\mathcal {M}_{\mathcal {Y}})$ and $j_{*}j^{*}X\in j_*(\mathcal {E}_{\mathcal {Z}})=j_*(\opname{Filt} \mathcal {M}_{\mathcal {Z}})\subseteq \opname{Filt} j_*(\mathcal {M}_{\mathcal {Z}})$. By Proposition \ref{filtration}, we have $X\in \opname{Filt} (i_*(\mathcal {M}_{\mathcal {Y}})\sqcup  j_*(\mathcal {M}_{\mathcal {Z}}))=\opname{Filt} \mathcal {M}_{\mathcal {X}}$. It follows that $\mathcal {E}_{\mathcal {X}}=\opname{Filt} \mathcal {M}_{\mathcal {X}}$.

$(2)\Rightarrow (1)$: Assume that $\mathcal {E}_{\mathcal {X}}$ is left Schur. We only prove that $\mathcal {E}_{\mathcal {Y}}$ is left Schur. The proof of $\mathcal {E}_{\mathcal {Z}}$ being left Schur when $\mathcal {E}_{\mathcal {X}}$ is left Schur is analogous. For any short exact sequence $0\longrightarrow L\longrightarrow M\longrightarrow N\longrightarrow 0$ in $\mathcal{Y}$ with $L\in \mathcal {E}_{\mathcal {Y}}$ and $N\in \mathcal {E}_{\mathcal {Y}}$, we have a short exact sequence $0\longrightarrow i_*L\longrightarrow i_*M\longrightarrow i_*N\longrightarrow 0$ in $\mathcal{X}$ since $i_*$ is exact. Thanks to Remark \ref{six-functor}, $i_*L\in \mathcal {E}_{\mathcal {X}}$ and $i_*N\in \mathcal {E}_{\mathcal {X}}$. Then $i_*M\in \mathcal {E}_{\mathcal {X}}$ since $\mathcal {E}_{\mathcal {X}}$ is extension-closed. By the definition of $\mathcal {E}_{\mathcal {X}}$, we have $M\in \mathcal {E}_{\mathcal {Y}}$. Thus $\mathcal {E}_{\mathcal {Y}}$ is extension-closed.

Let $Y$ be a simple object in $\mathcal {E}_{\mathcal {Y}}$ and $f: Y\longrightarrow Y'$ a morphism with $Y'\in \mathcal {E}_{\mathcal {Y}}$. By the definition of $\mathcal {E}_{\mathcal {X}}$, $i_*f: i_*Y\longrightarrow i_*Y'$ is a morphism in $\mathcal {E}_{\mathcal {X}}$. Suppose that
\begin{equation}
0\longrightarrow L\longrightarrow i_*Y\longrightarrow N\longrightarrow 0
\end{equation}
is a short exact sequence in $\mathcal {E}_{\mathcal {X}}$. Since $j^*$ is exact and $j^*i_*=0$, we have
$
j^*L=j^*N=0.
$
Hence $L,N\in \operatorname{Im}i_*$, so there exist $L',N'\in \mathcal Y$ such that
$
L\simeq i_*L', N\simeq i_*N'.
$
Since $L,N\in \mathcal {E}_{\mathcal {X}}$, it follows that $L',N'\in \mathcal {E}_{\mathcal {Y}}$. Applying the exact functor $i^!$ to (3.1), we obtain
$$
0\longrightarrow L'\longrightarrow Y\longrightarrow N'\longrightarrow 0.
$$
Since $Y$ is simple in $\mathcal {E}_{\mathcal {Y}}$, either $L'=0$ or $N'=0$. Therefore either $L=0$ or $N=0$, and hence $i_*Y$ is simple in $\mathcal {E}_{\mathcal {X}}$. Thus $i_*f$ is either zero or an injection since $\mathcal {E}_{\mathcal {X}}$ is left Schur. Since $i^!$ is exact, it follows that $f\cong i^!i_*f$ is either zero or an injection. We have finished to prove that $\mathcal {E}_{\mathcal {Y}}$ is left Schur.
\end{proof}
Now we show that the construction in Theorem \ref{main theorem} restricts to wide subcategories and torsion-free classes.
\begin{theorem}\label{theorem-wide}
Let $R(\mathcal {Y}, \mathcal {X}, \mathcal {Z})$ be a recollement of length abelian categories.
Assume that $\mathcal {W}_{\mathcal {Y}}$ and $\mathcal {W}_{\mathcal {Z}}$ are respectively subcategories of $\mathcal {Y}$ and $\mathcal {Z}$, $\mathcal {W}_{\mathcal {X}}=\{X\in \mathcal {X}|j^{*}X\in\mathcal {W}_{\mathcal {Z}}, i^{!}X\in\mathcal {W}_{\mathcal {Y}}\}$ is a subcategory of $\mathcal {X}$. If $i^{!}$ is exact, then the following are equivalent:
\begin{itemize}
\item[\rm(1)] $\mathcal {W}_{\mathcal {Y}}$ and $\mathcal {W}_{\mathcal {Z}}$ are wide.
\item[\rm(2)] $\mathcal {W}_{\mathcal {X}}$ is wide.
\end{itemize}
\end{theorem}
\begin{proof}
$(1)\Rightarrow (2)$: Assume that $\mathcal {W}_{\mathcal {Y}}$ and $\mathcal {W}_{\mathcal {Z}}$ are wide. By Theorem \ref{bijection}, $\mathcal {S}_{\mathcal {Y}}:=\opname{sim} \mathcal {W}_{\mathcal {Y}}\in \sbrick \mathcal {Y}$ and $\mathcal {S}_{\mathcal {Z}}:=\opname{sim} \mathcal {W}_{\mathcal {Z}}\in \sbrick \mathcal {Z}$. It follows that $\mathcal {W}_{\mathcal {Y}}=\opname{Filt} \mathcal {S}_{\mathcal {Y}}$ and $\mathcal {W}_{\mathcal {Z}}=\opname{Filt} \mathcal {S}_{\mathcal {Z}}$. By Proposition \ref{intermediate}, we have $\mathcal {S}_{\mathcal {X}}:=i_*(\mathcal {S}_{\mathcal {Y}})\sqcup j_*(\mathcal {S}_{\mathcal {Z}})\in \sbrick \mathcal {X}$. Using the same argument as in the proof of Theorem \ref{main theorem}, we can prove that $\mathcal {W}_{\mathcal {X}}=\opname{Filt} \mathcal {S}_{\mathcal {X}}$. Then $\mathcal {W}_{\mathcal {X}}$ is wide by Theorem \ref{bijection}.

$(2)\Rightarrow (1)$: Assume that $\mathcal {W}_{\mathcal {X}}$ is wide. We only prove that $\mathcal {W}_{\mathcal {Z}}$ is wide. The proof of $\mathcal {W}_{\mathcal {Y}}$ being wide when $\mathcal {W}_{\mathcal {X}}$ is wide is similar. Let $f': M'\rightarrow N'$ be a morphism in $\mathcal {W}_{\mathcal {Z}}$. We have the following exact sequence:
\begin{equation}
0\longrightarrow {\rm Ker}f'\longrightarrow M'\buildrel {f'} \over\longrightarrow N'\longrightarrow{\rm Coker}f'\longrightarrow 0.
\end{equation}
By Lemma \ref{exact-functor}, $j_*$ is exact. Applying the functor $j_*$ to (3.2), we have the following commutative diagram:
$$\xymatrix{
0\ar[r] & j_{*}({\rm Ker}f')\ar[r]\ar@{.>}[d] & j_{*}M'\ar[r]^{j_{*}f'}\ar[d]^{=} & j_{*}N'\ar[d]^{=}\ar[r] & j_{*}({\rm Coker}f')\ar[r]\ar@{.>}[d] & 0\\
0\ar[r] & {\rm Ker}j_{*}f'\ar[r] & j_*M'\ar[r]^{j_{*}f'} & j_*N'\ar[r] & {\rm Coker}j_{*}f'\ar[r] & 0
.}$$
By definition of $\mathcal {W}_{\mathcal {X}}$, $j_*M'\in \mathcal {W}_{\mathcal {X}}$ and $j_*N'\in \mathcal {W}_{\mathcal {X}}$. Then we have $j_{*}({\rm Ker}f')\cong {\rm Ker}j_{*}f'\in \mathcal {W}_{\mathcal {X}}$ and $j_{*}({\rm Coker}f')\cong{\rm Coker}j_{*}f'\in \mathcal {W}_{\mathcal {X}}$ since $\mathcal {W}_{\mathcal {X}}$ is wide. It follows that ${\rm Ker}f'\in \mathcal {W}_{\mathcal {Z}}$ and ${\rm Coker}f'\in\mathcal {W}_{\mathcal {Z}}$. Thus $\mathcal {W}_{\mathcal {Z}}$ is closed under kernels and cokernels.

For any short exact sequence $0\longrightarrow L'\longrightarrow M'\longrightarrow N'\longrightarrow 0$ in $\mathcal{Z}$ with $L'\in \mathcal {W}_{\mathcal {Z}}$ and $N'\in \mathcal {W}_{\mathcal {Z}}$, we have a short exact sequence $0\longrightarrow j_*L'\longrightarrow j_*M'\longrightarrow j_*N'\longrightarrow 0$ since $j_*$ is exact by Lemma \ref{exact-functor}. By applying $j^*$, we have $j_*L'\in \mathcal {W}_{\mathcal {X}}$ and $j_*N'\in \mathcal {W}_{\mathcal {X}}$. Then $j_*M'\in \mathcal {W}_{\mathcal {X}}$ since $\mathcal {W}_{\mathcal {X}}$ wide. By the definition of $\mathcal {W}_{\mathcal {X}}$, we have $M'\in \mathcal {W}_{\mathcal {Z}}$. We have finished to prove that $\mathcal {W}_{\mathcal {Z}}$ is wide.
\end{proof}
The following result about torsion-free classes can be deduced quickly from \cite{MH}. For this result, we don't assume that the functor $i^!$ is exact.
\begin{theorem}\label{theorem-torsion-free}
Let $R(\mathcal {Y}, \mathcal {X}, \mathcal {Z})$ be a recollement of length abelian categories.
Assume that $\mathcal {F}_{\mathcal {Y}}$ and $\mathcal {F}_{\mathcal {Z}}$ are respectively subcategories of $\mathcal {Y}$ and $\mathcal {Z}$, $\mathcal {F}_{\mathcal {X}}=\{X\in \mathcal {X}|j^{*}X\in\mathcal {F}_{\mathcal {Z}}, i^{!}X\in\mathcal {F}_{\mathcal {Y}}\}$ is a subcategory of $\mathcal {X}$. Then the following are equivalent:
\begin{itemize}
\item[\rm(1)] $\mathcal {F}_{\mathcal {Y}}$ and $\mathcal {F}_{\mathcal {Z}}$ are torsion-free classes.
\item[\rm(2)] $\mathcal {F}_{\mathcal {X}}$ is a torsion-free class.
\end{itemize}
\end{theorem}
\begin{proof}
$(1)\Rightarrow (2)$: By \cite[Theorem 1(1)]{MH}, this is immediate.

$(2)\Rightarrow (1)$: We show that $\mathcal {F}_{\mathcal {Y}}=i^!(\mathcal {F}_{\mathcal {X}})$ and $\mathcal {F}_{\mathcal {Z}}=j^*(\mathcal {F}_{\mathcal {X}})$. By the definition of $\mathcal {F}_{\mathcal {X}}$, it is obvious that $i^!(\mathcal {F}_{\mathcal {X}})\subset\mathcal {F}_{\mathcal {Y}}$. For any $Y\in \mathcal {F}_{\mathcal {Y}}$, $i_*Y\in \mathcal {F}_{\mathcal {X}}$. It follows that $Y\cong i^!i_*Y\in i^!(\mathcal {F}_{\mathcal {X}})$. Thus $\mathcal {F}_{\mathcal {Y}}=i^!(\mathcal {F}_{\mathcal {X}})$. By \cite[Theorem 2(1)]{MH}, $\mathcal {F}_{\mathcal {Y}}$ is torsion-free.

By the definition of $\mathcal {F}_{\mathcal {X}}$, it is obvious that $j^*(\mathcal {F}_{\mathcal {X}})\subset\mathcal {F}_{\mathcal {Z}}$. For any $Z\in \mathcal {F}_{\mathcal {Z}}$, $j_*Z\in \mathcal {F}_{\mathcal {X}}$. It follows that $Z\cong j^*j_*Z\in j^*(\mathcal {F}_{\mathcal {X}})$. Thus $\mathcal {F}_{\mathcal {Z}}=j^*(\mathcal {F}_{\mathcal {X}})$. For any $X\in \mathcal {F}_{\mathcal {X}}$, it is easy to check $j_*j^*X\in \mathcal {F}_{\mathcal {X}}$. So we have $j_*j^*(\mathcal {F}_{\mathcal {X}})\subset \mathcal {F}_{\mathcal {X}}$. By \cite[Theorem 2(2)]{MH}, $\mathcal {F}_{\mathcal {Z}}$ is torsion-free.
\end{proof}

In the following, we show that in Theorem \ref{theorem-torsion-free}, functorially finiteness of $\mathcal {F}_{\mathcal {X}}$ implies functorially finiteness of $\mathcal {F}_{\mathcal {Y}}$ and $\mathcal {F}_{\mathcal {Z}}$. However, the converse is not true in general. A counterexample is given by symmetric ladders of abelian categories, see \cite[Theorem 3.3]{Z}.

\begin{lemma}\label{functorially}
Let $R(\mathcal {Y}, \mathcal {X}, \mathcal {Z})$ be a recollement of length abelian categories.
Assume that $\mathcal {F}_{\mathcal {Y}}$ and $\mathcal {F}_{\mathcal {Z}}$ are respectively subcategories of $\mathcal {Y}$ and $\mathcal {Z}$, and consider $\mathcal {F}_{\mathcal {X}}=\{X\in \mathcal {X}|j^{*}X\in\mathcal {F}_{\mathcal {Z}}, i^{!}X\in\mathcal {F}_{\mathcal {Y}}\}$ is a subcategory of $\mathcal {X}$. Then we have the following:
\begin{itemize}
\item[\rm(1)] if $\mathcal {F}_{\mathcal {X}}$ is covariantly finite, then $\mathcal {F}_{\mathcal {Y}}$ and $\mathcal {F}_{\mathcal {Z}}$ are covariantly finite.
\item[\rm(2)] if $\mathcal {F}_{\mathcal {X}}$ is contravariantly finite, then $\mathcal {F}_{\mathcal {Y}}$ and $\mathcal {F}_{\mathcal {Z}}$ are contravariantly finite.
\item[\rm(3)] if $\mathcal {F}_{\mathcal {X}}$ is functorially finite, then $\mathcal {F}_{\mathcal {Y}}$ and $\mathcal {F}_{\mathcal {Z}}$ are functorially finite.
\end{itemize}
\end{lemma}
\begin{proof}
By the proof of necessity in \cite[Proposition 3.1]{Z}, we immediately get (1). Dually we can get (2). Combining (1) and (2), we get (3).
\end{proof}
By Theorem \ref{theorem-torsion-free} and Lemma \ref{functorially}, we have the following result.
\begin{corollary}
Let $R(\mathcal {Y}, \mathcal {X}, \mathcal {Z})$ be a recollement of length abelian categories.
Assume that $\mathcal {F}_{\mathcal {Y}}$ and $\mathcal {F}_{\mathcal {Z}}$ are respectively subcategories of $\mathcal {Y}$ and $\mathcal {Z}$, and consider $\mathcal {F}_{\mathcal {X}}=\{X\in \mathcal {X}|j^{*}X\in\mathcal {F}_{\mathcal {Z}}, i^{!}X\in\mathcal {F}_{\mathcal {Y}}\}$ is a subcategory of $\mathcal {X}$. If $\mathcal {F}_{\mathcal {X}}$ is functorially finite torsion-free, then $\mathcal {F}_{\mathcal {Y}}$ and $\mathcal {F}_{\mathcal {Z}}$ are functorially finite torsion-free.
\end{corollary}

Using the bijection between torsion-free classes and cofinally closed monobricks in Theorem \ref{bijection}, we can give an explicit construction of cofinally closed monobricks in recollements. More precisely, if $i^!$ is exact, the construction of monobricks in Proposition \ref{intermediate}(4) restricts to cofinally closed monobricks.

\begin{theorem}
Let $R(\mathcal {Y}, \mathcal {X}, \mathcal {Z})$ be a recollement of length abelian categories. Assume that $\mathcal {M}_\mathcal {Y}\in \mbrick_{\rm{c.c.}} \mathcal{Y}$ and $\mathcal {M}_\mathcal {Z}\in \mbrick_{\rm{c.c.}} \mathcal{Z}$. If $i^!$ is exact, then $\mathcal {M}_\mathcal {X}:=i_*(\mathcal {M}_\mathcal {Y})\sqcup j_*(\mathcal {M}_\mathcal {Z})\in \mbrick_{\rm{c.c.}} \mathcal{X}$.
\end{theorem}
\begin{proof}
By Theorem \ref{bijection}, $\mathcal {F}_{\mathcal {Y}}:=\opname{Filt} \mathcal {M}_\mathcal {Y}\in \opname{torf} \mathcal {Y}$ and $\mathcal {F}_{\mathcal {Z}}:=\opname{Filt}\mathcal {M}_\mathcal {Z}\in \opname{torf}\mathcal {Z}$. It follows that $\mathcal {M}_\mathcal {Y}=\opname{sim} \mathcal {F}_{\mathcal {Y}}$ and $\mathcal {M}_\mathcal {Z}=\opname{sim} \mathcal {F}_{\mathcal {Z}}$. By Theorem \ref{theorem-torsion-free}, we have $\mathcal {F}_{\mathcal {X}}:=\{X|j^*X\in \mathcal {F}_{\mathcal {Z}}, i^!X\in \mathcal {F}_{\mathcal {Y}}\}\in \opname{torf}{\mathcal {X}}$. Using the same argument as in the proof of Theorem \ref{main theorem}, we can prove that $\mathcal {F}_{\mathcal {X}}=\opname{Filt} \mathcal {M}_{\mathcal {X}}$. Then $\mathcal {M}_\mathcal {X}\in \mbrick_{\text{c.c.}} \mathcal{X}$ by Theorem \ref{bijection}.
\end{proof}

\section{Applications}\label{sec:4}

In this section, we will apply the results to a special case of the MacPherson-Vilonen construction of abelian recollements and triangular matrix algebras. The MacPherson-Vilonen construction of abelian recollements, given below, first appeared in \cite[Section 1]{MV} in order to simplify the definition of the category of perverse sheaves over a stratified space.

\begin{definition}\label{MV-def}
Let $\mathcal {Y}$ and $\mathcal {Z}$ be abelian categories. Let $F: \mathcal {Z}\rightarrow \mathcal {Y}$ be a right exact functor, let $G: \mathcal {Z}\rightarrow \mathcal {Y}$ be a left exact functor, and let $\alpha: F\rightarrow G$ be a natural transformation. The \emph{MacPherson-Vilonen construction} for $\alpha $ is a recollement of abelian categories
$$\xymatrix@!C=2pc{
\mathcal {Y}\; \ar@{>->}[rr]|{i_{*}} && \mathcal {X(\alpha)} \ar@<-4.0mm>@{->>}[ll]_{i^{*}} \ar@{->>}[rr]|{j^{*}} \ar@{->>}@<4.0mm>[ll]^{i^{!}}&& \mathcal {Z} \ar@{>->}@<-4.0mm>[ll]_{j_{!}} \ar@{>->}@<4.0mm>[ll]^{j_{*}}
}$$
given by the following data.
\begin{enumerate}
\item The abelian category $\mathcal {X(\alpha)}$, which has objects and morphisms as follows.
\begin{itemize}
\item Objects $(Z,Y,g,f)$ given by an object $Z\in\mathcal {Z}$, an object $Y\in\mathcal {Y}$ and morphisms
\begin{displaymath}
\xymatrix{FZ\ar[r]^{f}&Y\ar[r]^{g}&GZ}
\end{displaymath}
such that $gf=\alpha_Z$.
\item Morphisms $(z,y):(Z,Y,g,f)\to(Z',Y',g',f')$ given by morphisms $z:Z\to Z'\in\mathcal {Z}$ and $y:Y\to Y'\in\mathcal {Y}$ such that the diagram
\begin{displaymath}
\xymatrix{FZ\ar[d]_{Fz}\ar[r]^{f}&Y\ar[d]^y\ar[r]^{g}&GZ\ar[d]^{Gz}
\\FZ'\ar[r]^{f'}&Y'\ar[r]^{g'}&GZ'}
\end{displaymath}
commutes.
\end{itemize}
\item The functors $j^*$, $i^*$ and $i^!$ are defined by
$$
j^*(Z,Y,g,f)=Z,\quad i^*(Z,Y,g,f)={\rm Coker}f,\quad i^!(Z,Y,g,f)={\rm ker}g
$$
for each $(Z,Y,g,f)\in\mathcal {X(\alpha)}$.
\item The functor $i_*$ is defined by $i_*Y=(0,Y,0,0)$ for each $Y\in\mathcal {Y}$.
\item The functors $j_!$ and $j_*$ are defined by
$$
j_!Z=(Z,FZ,\alpha_Z,1_{FZ}),\quad
j_*Z=(Z,GZ,1_{GZ},\alpha_Z)
$$
for each $Z\in\mathcal {Z}$.
\end{enumerate}
\end{definition}
Consider one particular case of this construction, see \cite[Section 3.2]{FP}. Let $F: \mathcal {Z}\rightarrow \mathcal {Y}$ be a right exact functor. Take $\alpha: F\rightarrow 0$ to be the transformation into the trivial functor. The corresponding construction is denoted by $\mathcal {Y}\rtimes _{F} \mathcal {Z}$. Thus objects of this category are triples $(Z,Y,f)$, where $Z$ and $Y$ are objects of $\mathcal {Z}$ and $\mathcal {Y}$ respectively, and $f$ is a morphism $f: FZ\rightarrow Y$. Moreover, $i^!$ is exact and $i^{!}(Z,Y,f)=Y$.
\begin{corollary}
Let $\mathcal {Y}$ and $\mathcal {Z}$ be length abelian categories.  Let $\mathcal {E}_{\mathcal {Y}}\subset \mathcal {Y}$, $\mathcal {E}_{\mathcal {Z}}\subset \mathcal {Z}$. Take $\mathcal {E}_{\mathcal {X}}$ the full subcategory of $\mathcal {Y}\rtimes _{F} \mathcal {Z}$ consisting of objects of the form $(Z,Y,f)$ where $Z\in \mathcal {E}_{\mathcal {Z}}$ and $Y\in \mathcal {E}_{\mathcal {Y}}$. Then $\mathcal {E}_{\mathcal {Y}}$ and $\mathcal {E}_{\mathcal {Z}}$ are left Schur (resp. wide, torsion-free) if and only if $\mathcal {E}_{\mathcal {X}}$ is left Schur (resp. wide, torsion-free).
\end{corollary}
\begin{proof}
By Definition \ref{MV-def}, there is a recollement of $\mathcal {Y}\rtimes _{F} \mathcal {Z}$ relative to $\mathcal {Y}$ and $\mathcal {Z}$, $j^*(Z,Y,f)=Z$ and $i^{!}(Z,Y,f)=Y$. Every $(Z,Y,f)\in \mathcal{Y}\rtimes_F \mathcal{Z}$ fits into a short exact sequence
\[
0\longrightarrow (0,Y,0)\longrightarrow (Z,Y,f)\longrightarrow (Z,0,0)\longrightarrow 0.
\]
Since $(0,Y,0)$ and $(Z,0,0)$ have finite length whenever $Y$ and $Z$ do, it follows that every object of $\mathcal {Y}\rtimes _{F} \mathcal {Z}$ has finite length. Hence, if $\mathcal{Y}$ and $\mathcal{Z}$ are length abelian categories, then
$\mathcal{Y}\rtimes_F \mathcal{Z}$ is also a length abelian category. Then by Theorems \ref{main theorem}, \ref{theorem-wide} and \ref{theorem-torsion-free}, we can get the result immediately .
\end{proof}

From now on, we fix a field $K$. Let $\Lambda$ be a finite-dimensional $K$-algebra. We denote by $\opname{mod}\Lambda$ the category of finitely generated right $\Lambda$-modules. We refer to \cite{ARS, ASS} for basic terminology of representation theory of finite-dimensional algebras. In the following, we will apply our results to triangular matrix algebras.

Let $B$ and $C$ be two finite-dimensional $K$-algebras and $M$ a finitely generated $C$-$B$-bimodule. Consider the matrix algebra $A=\left(\begin{array}{cc} B & 0\\ {}_CM_B & C\end{array}\right)$, see \cite[Pages 73-75]{ARS}. A right $A$-module is identified with a triple $(X, Y)_{f}$, or simply $(X, Y)$ if $f$ is clear, where $X$ is a right $B$-module, $Y$ is a right $C$-module, and $f: Y\otimes _{C}M\longrightarrow X$ is a right $B$-map. Put $e_1=\left(\begin{array}{cc} 1_{B} & 0\\ 0 & 0\end{array}\right)$ and $e_2=\left(\begin{array}{cc} 0 & 0\\ 0 & 1_{C}\end{array}\right)$. It is well known that there is a recollement of $\mod A$ relative to $\mod B$ and $\mod C$ as follows:
$$\xymatrix@!C=2pc{
\mod B\; \ar@{>->}[rr]|{i_{*}} && \mod A \ar@<-4.0mm>@{->>}[ll]_{i^{*}} \ar@{->>}[rr]|{?\otimes _A Ae_2} \ar@{->>}@<4.0mm>[ll]^{?\otimes _A Ae_1}&& \mod C, \ar@{>->}@<-4.0mm>[ll]_{j_{!}} \ar@{>->}@<4.0mm>[ll]^{j_{*}}
}$$
where $i^!=?\otimes _A Ae_1$ is exact, see \cite[Section 4.5]{FZ}, \cite[Example 2.10]{PV}.

\begin{corollary}
Let $\mathcal {E}_{B}\subset \mod B$, $\mathcal {E}_{C}\subset \mod C$. Take $\mathcal {E}_{A}$ the full subcategory of $\mod A$ consisting of the modules of the form $(X, Y)_{f}$ where $X\in \mathcal {E}_{B}$ and $Y\in \mathcal {E}_{C}$. Then $\mathcal {E}_{B}$ and $\mathcal {E}_{C}$ are left Schur (resp. wide, torsion-free) if and only if $\mathcal {E}_{A}$ is left Schur (resp. wide, torsion-free).
\end{corollary}

\begin{proof}
By Theorems \ref{main theorem}, \ref{theorem-wide} and \ref{theorem-torsion-free}, we only have to prove that $$\mathcal {E}_{A}=\{T\in \mod A|T\otimes _A Ae_2\in \mathcal {E}_{\mathcal {C}}, T\otimes _A Ae_1\in \mathcal {E}_{\mathcal {B}}\}.$$ Assume that $T=(X, Y)_{f}$. Then $T\otimes _A Ae_2\cong Te_{2}\cong Y$ and $T\otimes _A Ae_1\cong Te_{1}\cong X$. The result is immediate.
\end{proof}
Finally we give examples to illustrate our results. For a finite-dimensional algebra $\Lambda$, we put $\opname{Schur_{L}}\Lambda:=\opname{Schur_{L}}(\mod \Lambda )$.
\begin{example}\label{exam-1}
Let $B$ be the path algebra of quiver $2\longrightarrow 3$ and $C$ the path algebra of quiver with only one-vertex $1$ and no arrows. Take $M=P_{2}^B$, the indecomposable projective right $B$-module corresponding to the vertex 2, regarded as a $C$-$B$-bimodule via the scalar action of $C=K$. Then the triangular matrix algebra $A$ is the path algebra of quiver $1\longrightarrow 2\longrightarrow 3$.

The Auslander-Reiten quiver of $\mod B$ is
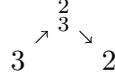
$$\begin{tikzpicture}[scale=.6]
      [baseline={([yshift=-.5ex]current bounding box.center)}, scale=0.4,  every node/.style={scale=0.5}]
      \node (1) at (0,0)  {$3$};;
      \node (21) at (1,1)  {$\substack{2\\3}$};
      \node (2) at (2,0) {$2$};

      \draw[->] (1) -- (21);
      \draw[->] (21) -- (2);
\end{tikzpicture}$$

The Auslander-Reiten quiver of $\mod A$.
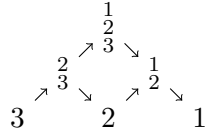
$$
\begin{tikzpicture}[scale=.6]
      \node (1) at (0,0)  {$3$};
      \node (21) at (1,1)  {$\substack{2\\3}$};
      \node (2) at (2,0) {$2$};
      \node (321) at (2,2) {$\substack{1\\2\\3}$};
      \node (32) at (3,1) {$\substack{1\\2}$};
      \node (3) at (4,0) {$1$};
      \draw[->] (1) -- (21);
      \draw[->] (21) -- (2);
      \draw[->] (21) -- (321);
      \draw[->] (321) -- (32);
      \draw[->] (2) -- (32);
      \draw[->] (32) -- (3);
\end{tikzpicture}
$$
In Table 1, we list left Schur subcategories of $\mod A$ constructed from that of $\mod B$ and $\mod C$. We write the left Schur subcategories $\opname{Filt} \mathcal {M}$ in the AR quiver, where the black vertices are the corresponding monobricks $\mathcal {M}$, and the white vertices denote the remaining objects in $\opname{Filt} \mathcal {M}$. The subcategories which are not torsion-free are red in color. The subcategories which are not wide are blue in color. In the middle, we construct 12 left Schur subcategories of $\mod A$ by Theorem \ref{main theorem}.
\begin{table}[htbp]
  \caption{\mbox{}\label{tab:table1}}
  \vspace{-0.35em}
  \begin{center}
 \begin{tabular}{lcl}
  \toprule
 $ \opname{Schur_{L}}B $\quad\quad\quad & $\opname{Schur_{L}}A$ \quad\quad\quad & $\opname{Schur_{L}}C$ \\
  \midrule
$\begin{tikzpicture}
      [baseline={([yshift=-.5ex]current bounding box.center)}, scale=0.4,  every node/.style={scale=0.5}]
      \node (1) at (0,0)  {};
      \node (21) at (1,1)  {};
      \node (2) at (2,0) {};

      \draw[->] (1) -- (21);
      \draw[->] (21) -- (2);
    \end{tikzpicture}$ &
$\begin{tikzpicture}
      [baseline={([yshift=-.5ex]current bounding box.center)}, scale=0.4,  every node/.style={scale=0.5}]
      \node (1) at (0,0) {};
      \node (21) at (1,1)  {};
      \node (2) at (2,0) {};
      \node (321) at (2,2) {};
      \node (32) at (3,1) {};
      \node (3) at (4,0) {};
      \node at (0,-.2) {};
      \node at (0,2.2) {};

      \draw[->] (1) -- (21);
      \draw[->] (21) -- (2);
      \draw[->] (21) -- (321);
      \draw[->] (321) -- (32);
      \draw[->] (2) -- (32);
      \draw[->] (32) -- (3);
    \end{tikzpicture}$ \quad\quad
   & \quad\quad
  $\O $ \\
\midrule
$\begin{tikzpicture}
      [baseline={([yshift=-.5ex]current bounding box.center)}, scale=0.4,  every node/.style={scale=0.5}]
      \node (1) at (0,0)  {};
      \node (21) at (1,1)  {};
      \node (2) at (2,0) {};

      \draw[->] (1) -- (21);
      \draw[->] (21) -- (2);
    \end{tikzpicture}$ &
$\begin{tikzpicture}
      [baseline={([yshift=-.5ex]current bounding box.center)}, scale=0.4,  every node/.style={scale=0.5}]
      \node (1) at (0,0) {};
      \node (21) at (1,1)  {};
      \node (2) at (2,0) {};
      \node (321) at (2,2) {};
      \node (32) at (3,1) {};
      \node (3) at (4,0) [black] {};
      \node at (0,-.2) {};
      \node at (0,2.2) {};

      \draw[->] (1) -- (21);
      \draw[->] (21) -- (2);
      \draw[->] (21) -- (321);
      \draw[->] (321) -- (32);
      \draw[->] (2) -- (32);
      \draw[->] (32) -- (3);
    \end{tikzpicture}$ \quad\quad
   & \quad\quad
  $\bullet  $ \\
\midrule
$\begin{tikzpicture}
      [baseline={([yshift=-.5ex]current bounding box.center)}, scale=0.4,  every node/.style={scale=0.5}]
      \node (1) at (0,0)  {};
      \node (21) at (1,1)  {};
      \node (2) at (2,0)[black] {};

      \draw[->] (1) -- (21);
      \draw[->] (21) -- (2);
    \end{tikzpicture}$ &
$\begin{tikzpicture}
      [baseline={([yshift=-.5ex]current bounding box.center)}, scale=0.4,  every node/.style={scale=0.5}]
      \node (1) at (0,0) {};
      \node (21) at (1,1)  {};
      \node (2) at (2,0)[black] {};
      \node (321) at (2,2) {};
      \node (32) at (3,1) {};
      \node (3) at (4,0) {};
      \node at (0,-.2) {};
      \node at (0,2.2) {};

      \draw[->] (1) -- (21);
      \draw[->] (21) -- (2);
      \draw[->] (21) -- (321);
      \draw[->] (321) -- (32);
      \draw[->] (2) -- (32);
      \draw[->] (32) -- (3);
    \end{tikzpicture}$ \quad\quad
   & \quad\quad
  $\O $ \\
\midrule
$\begin{tikzpicture}
      [baseline={([yshift=-.5ex]current bounding box.center)}, scale=0.4,  every node/.style={scale=0.5}]
      \node (1) at (0,0)  {};
      \node (21) at (1,1)  {};
      \node (2) at (2,0)[black] {};

      \draw[->] (1) -- (21);
      \draw[->] (21) -- (2);
    \end{tikzpicture}$ &
$\begin{tikzpicture}
      [baseline={([yshift=-.5ex]current bounding box.center)}, scale=0.4,  every node/.style={scale=0.5}]
      \node (1) at (0,0) {};
      \node (21) at (1,1)  {};
      \node (2) at (2,0)[black] {};
      \node (321) at (2,2) {};
      \node (32) at (3,1)[white] {};
      \node (3) at (4,0) [black]{};
      \node at (0,-.2) {};
      \node at (0,2.2) {};

      \draw[->] (1) -- (21);
      \draw[->] (21) -- (2);
      \draw[->] (21) -- (321);
      \draw[->] (321) -- (32);
      \draw[->] (2) -- (32);
      \draw[->] (32) -- (3);
    \end{tikzpicture}$ \quad\quad
   & \quad\quad
  $\bullet  $ \\
\midrule
$\begin{tikzpicture}
      [baseline={([yshift=-.5ex]current bounding box.center)}, scale=0.4,  every node/.style={scale=0.5}]
      \node (1) at (0,0)[black]  {};
      \node (21) at (1,1)  {};
      \node (2) at (2,0) {};

      \draw[->] (1) -- (21);
      \draw[->] (21) -- (2);
    \end{tikzpicture}$ &
$\begin{tikzpicture}
      [baseline={([yshift=-.5ex]current bounding box.center)}, scale=0.4,  every node/.style={scale=0.5}]
      \node (1) at (0,0)[black] {};
      \node (21) at (1,1)  {};
      \node (2) at (2,0) {};
      \node (321) at (2,2) {};
      \node (32) at (3,1) {};
      \node (3) at (4,0) {};
      \node at (0,-.2) {};
      \node at (0,2.2) {};

      \draw[->] (1) -- (21);
      \draw[->] (21) -- (2);
      \draw[->] (21) -- (321);
      \draw[->] (321) -- (32);
      \draw[->] (2) -- (32);
      \draw[->] (32) -- (3);
    \end{tikzpicture}$ \quad\quad
   & \quad\quad
  $\O $ \\
\midrule
$\begin{tikzpicture}
      [baseline={([yshift=-.5ex]current bounding box.center)}, scale=0.4,  every node/.style={scale=0.5}]
      \node (1) at (0,0) [black] {};
      \node (21) at (1,1)  {};
      \node (2) at (2,0) {};

      \draw[->] (1) -- (21);
      \draw[->] (21) -- (2);
    \end{tikzpicture}$ &
$\begin{tikzpicture}
      [baseline={([yshift=-.5ex]current bounding box.center)}, scale=0.4,  every node/.style={scale=0.5}]
      \node (1) at (0,0)[black] {};
      \node (21) at (1,1)  {};
      \node (2) at (2,0) {};
      \node (321) at (2,2) {};
      \node (32) at (3,1) {};
      \node (3) at (4,0)[black] {};
      \node at (0,-.2) {};
      \node at (0,2.2) {};

      \draw[->] (1) -- (21);
      \draw[->] (21) -- (2);
      \draw[->] (21) -- (321);
      \draw[->] (321) -- (32);
      \draw[->] (2) -- (32);
      \draw[->] (32) -- (3);
    \end{tikzpicture}$ \quad\quad
   & \quad\quad
  $\bullet  $ \\
\midrule
\color{red}{$\begin{tikzpicture}
      [baseline={([yshift=-.5ex]current bounding box.center)}, scale=0.4,  every node/.style={scale=0.5}]
      \node (1) at (0,0)  {};
      \node (21) at (1,1) [black] {};
      \node (2) at (2,0) {};

      \draw[->] (1) -- (21);
      \draw[->] (21) -- (2);
    \end{tikzpicture}$} &
\color{red}{$\begin{tikzpicture}
      [baseline={([yshift=-.5ex]current bounding box.center)}, scale=0.4,  every node/.style={scale=0.5}]
      \node (1) at (0,0) {};
      \node (21) at (1,1) [black] {};
      \node (2) at (2,0) {};
      \node (321) at (2,2) {};
      \node (32) at (3,1) {};
      \node (3) at (4,0) {};
      \node at (0,-.2) {};
      \node at (0,2.2) {};

      \draw[->] (1) -- (21);
      \draw[->] (21) -- (2);
      \draw[->] (21) -- (321);
      \draw[->] (321) -- (32);
      \draw[->] (2) -- (32);
      \draw[->] (32) -- (3);
    \end{tikzpicture}$} \quad\quad
   & \quad\quad
 $\O $ \\
\midrule
\color{red}{$\begin{tikzpicture}
      [baseline={([yshift=-.5ex]current bounding box.center)}, scale=0.4,  every node/.style={scale=0.5}]
      \node (1) at (0,0)  {};
      \node (21) at (1,1)[black]  {};
      \node (2) at (2,0) {};

      \draw[->] (1) -- (21);
      \draw[->] (21) -- (2);
    \end{tikzpicture}$} &
\color{red}{$\begin{tikzpicture}
      [baseline={([yshift=-.5ex]current bounding box.center)}, scale=0.4,  every node/.style={scale=0.5}]
      \node (1) at (0,0) {};
      \node (21) at (1,1)[black]  {};
      \node (2) at (2,0) {};
      \node (321) at (2,2)[white] {};
      \node (32) at (3,1) {};
      \node (3) at (4,0) [black]{};
      \node at (0,-.2) {};
      \node at (0,2.2) {};

      \draw[->] (1) -- (21);
      \draw[->] (21) -- (2);
      \draw[->] (21) -- (321);
      \draw[->] (321) -- (32);
      \draw[->] (2) -- (32);
      \draw[->] (32) -- (3);
    \end{tikzpicture}$} \quad\quad
   & \quad\quad
  $\bullet  $ \\
\midrule
\color{blue}{$\begin{tikzpicture}
      [baseline={([yshift=-.5ex]current bounding box.center)}, scale=0.4,  every node/.style={scale=0.5}]
      \node (1) at (0,0) [black] {};
      \node (21) at (1,1) [black] {};
      \node (2) at (2,0) {};

      \draw[->] (1) -- (21);
      \draw[->] (21) -- (2);
    \end{tikzpicture}$} &
\color{blue}{$\begin{tikzpicture}
      [baseline={([yshift=-.5ex]current bounding box.center)}, scale=0.4,  every node/.style={scale=0.5}]
      \node (1) at (0,0) [black] {};
      \node (21) at (1,1)[black]  {};
      \node (2) at (2,0) {};
      \node (321) at (2,2) {};
      \node (32) at (3,1) {};
      \node (3) at (4,0) {};
      \node at (0,-.2) {};
      \node at (0,2.2) {};

      \draw[->] (1) -- (21);
      \draw[->] (21) -- (2);
      \draw[->] (21) -- (321);
      \draw[->] (321) -- (32);
      \draw[->] (2) -- (32);
      \draw[->] (32) -- (3);
    \end{tikzpicture}$} \quad\quad
   & \quad\quad
$\O $ \\
\midrule
\color{blue}{$\begin{tikzpicture}
      [baseline={([yshift=-.5ex]current bounding box.center)}, scale=0.4,  every node/.style={scale=0.5}]
      \node (1) at (0,0) [black] {};
      \node (21) at (1,1)[black]  {};
      \node (2) at (2,0) {};

      \draw[->] (1) -- (21);
      \draw[->] (21) -- (2);
    \end{tikzpicture}$} &
\color{blue}{$\begin{tikzpicture}
      [baseline={([yshift=-.5ex]current bounding box.center)}, scale=0.4,  every node/.style={scale=0.5}]
      \node (1) at (0,0)[black] {};
      \node (21) at (1,1) [black] {};
      \node (2) at (2,0) {};
      \node (321) at (2,2)[white] {};
      \node (32) at (3,1) {};
      \node (3) at (4,0)[black] {};
      \node at (0,-.2) {};
      \node at (0,2.2) {};

      \draw[->] (1) -- (21);
      \draw[->] (21) -- (2);
      \draw[->] (21) -- (321);
      \draw[->] (321) -- (32);
      \draw[->] (2) -- (32);
      \draw[->] (32) -- (3);
    \end{tikzpicture}$} \quad\quad
   & \quad\quad
  $\bullet  $ \\
\midrule
$\begin{tikzpicture}
      [baseline={([yshift=-.5ex]current bounding box.center)}, scale=0.4,  every node/.style={scale=0.5}]
      \node (1) at (0,0) [black] {};
      \node (21) at (1,1) [white] {};
      \node (2) at (2,0) [black]{};

      \draw[->] (1) -- (21);
      \draw[->] (21) -- (2);
    \end{tikzpicture}$ &
$\begin{tikzpicture}
      [baseline={([yshift=-.5ex]current bounding box.center)}, scale=0.4,  every node/.style={scale=0.5}]
      \node (1) at (0,0)[black] {};
      \node (21) at (1,1) [white] {};
      \node (2) at (2,0)[black] {};
      \node (321) at (2,2) {};
      \node (32) at (3,1) {};
      \node (3) at (4,0) {};
      \node at (0,-.2) {};
      \node at (0,2.2) {};

      \draw[->] (1) -- (21);
      \draw[->] (21) -- (2);
      \draw[->] (21) -- (321);
      \draw[->] (321) -- (32);
      \draw[->] (2) -- (32);
      \draw[->] (32) -- (3);
    \end{tikzpicture}$ \quad\quad
   & \quad\quad
  $\O $ \\
\midrule
$\begin{tikzpicture}
      [baseline={([yshift=-.5ex]current bounding box.center)}, scale=0.4,  every node/.style={scale=0.5}]
      \node (1) at (0,0) [black] {};
      \node (21) at (1,1) [white] {};
      \node (2) at (2,0)[black] {};

      \draw[->] (1) -- (21);
      \draw[->] (21) -- (2);
    \end{tikzpicture}$ &
$\begin{tikzpicture}
      [baseline={([yshift=-.5ex]current bounding box.center)}, scale=0.4,  every node/.style={scale=0.5}]
      \node (1) at (0,0) [black]{};
      \node (21) at (1,1) [white] {};
      \node (2) at (2,0)[black] {};
      \node (321) at (2,2)[white] {};
      \node (32) at (3,1)[white] {};
      \node (3) at (4,0)[black] {};
      \node at (0,-.2) {};
      \node at (0,2.2) {};

      \draw[->] (1) -- (21);
      \draw[->] (21) -- (2);
      \draw[->] (21) -- (321);
      \draw[->] (321) -- (32);
      \draw[->] (2) -- (32);
      \draw[->] (32) -- (3);
    \end{tikzpicture}$ \quad\quad
   & \quad\quad
  $\bullet  $ \\
  \bottomrule
 \end{tabular}
 \end{center}
\end{table}

\end{example}

\begin{example}
Let $B$ be the path algebra of quiver $3\longrightarrow 4$ and $C$ the path algebra of quiver $1\longrightarrow 2$. Take $M=_{C}P_2\otimes_K P_3^B$, where $_{C}P_2$ is the indecomposable projective left $C$-module corresponding to the vertex 2, and $ P_3^B$
is the indecomposable projective right $B$-module corresponding to the vertex 3. Then the triangular matrix algebra $A$ is the path algebra of quiver $1\longrightarrow 2\longrightarrow 3\longrightarrow 4$.
The Auslander-Reiten quiver of $\mod B$ is
$$\begin{tikzpicture}[scale=.6]
      [baseline={([yshift=-.5ex]current bounding box.center)}, scale=0.4,  every node/.style={scale=0.5}]
      \node (1) at (0,0)  {$4$};;
      \node (21) at (1,1)  {$\substack{3\\4}$};
      \node (2) at (2,0) {$3$};

      \draw[->] (1) -- (21);
      \draw[->] (21) -- (2);
\end{tikzpicture}$$
The Auslander-Reiten quiver of $\mod C$ is
$$\begin{tikzpicture}[scale=.6]
      [baseline={([yshift=-.5ex]current bounding box.center)}, scale=0.4,  every node/.style={scale=0.5}]
      \node (1) at (0,0)  {$2$};;
      \node (21) at (1,1)  {$\substack{1\\2}$};
      \node (2) at (2,0) {$1$};

      \draw[->] (1) -- (21);
      \draw[->] (21) -- (2);
\end{tikzpicture}$$
The Auslander-Reiten quiver of $\mod A$ is
$$
\begin{tikzpicture}[scale=.6]
      \node (1) at (0,0)  {$4$};
      \node (21) at (1,1)  {$\substack{3\\4}$};
      \node (2) at (2,0) {$3$};
      \node (321) at (2,2) {$\substack{2\\3\\4}$};
      \node (32) at (3,1) {$\substack{2\\3}$};
      \node (4321) at (3,3) {$\substack{1\\2\\3\\4}$};
      \node (432) at (4,2) {$\substack{1\\2\\3}$};
      \node (43) at (5,1) {$\substack{1\\2}$};
      \node (4) at (6,0) {$\substack{1}$};
      \node (3) at (4,0) {$2$};
      \draw[->] (1) -- (21);
      \draw[->] (21) -- (2);
      \draw[->] (21) -- (321);
      \draw[->] (321) -- (4321);
      \draw[->] (321) -- (32);
      \draw[->] (2) -- (32);
      \draw[->] (32) -- (3);
      \draw[->] (32) -- (432);
      \draw[->] (4321) -- (432);
      \draw[->] (432) -- (43);
      \draw[->] (3) -- (43);
      \draw[->] (43) -- (4);
\end{tikzpicture}
$$

Since both $B$ and $C$ are path algebras of type $A_2$, and a path algebra of type $A_2$ has six left Schur subcategories (as computed in Example \ref{exam-1}), we have $\#\opname{Schur_{L}}B=\#\opname{Schur_{L}}C=6$. By Theorem \ref{main theorem}, we construct 36 left Schur subcategories of $\mod A$ and divide them into the following cases:

(1) $\mathcal {E}_{B}=\O$ and $\mathcal {E}_{C}=\O$, then $\mathcal {E}_{A}=\O$.

(2) $\mathcal {E}_{B}=\O$ and $\mathcal {E}_{C}\neq \O$, we can construct 5 left Schur subcategories of $\mod A$.

(3) $\mathcal {E}_{B}\neq\O$ and $\mathcal {E}_{C}= \O$, we can construct 5 left Schur subcategories of $\mod A$.

(4) $\mathcal {E}_{B}=\mod B$ and $\mathcal {E}_{C}=\mod C$, then $\mathcal {E}_{A}=\mod A$.

(5) All remaining left Schur subcategories except for cases (1)-(4).

In table 2, we only list case (5), since cases (1)-(4) are easy to calculate. We write the left Schur subcategories $\opname{Filt} \mathcal {M}$ in the AR quiver, where the black vertices are the corresponding monobricks $\mathcal {M}$, and the white vertices denote the remaining objects in $\opname{Filt} \mathcal {M}$. The subcategories which are wide but not torsion-free are red in color. The subcategories which are torsion-free but not wide are blue in color. The subcategories which are neither torsion-free nor wide are yellow in color.
\begingroup
\setlength{\LTleft}{\fill}
\setlength{\LTright}{\fill}
\setlength{\LTpre}{0.8\baselineskip}
\setlength{\LTpost}{0pt}
\begin{longtable}{@{}lcl@{}}
\captionsetup{labelsep=period}
\caption{\mbox{}}\label{tab:table2}\\[-0.25em]
\toprule
$ \opname{Schur_{L}}B $\quad\quad\quad & $\opname{Schur_{L}}A$ \quad\quad\quad & $\opname{Schur_{L}}C$ \\
\midrule
\endfirsthead
\multicolumn{3}{c}{\tablename~\thetable\ (continued)}\\[0.35em]
\toprule
$ \opname{Schur_{L}}B $\quad\quad\quad & $\opname{Schur_{L}}A$ \quad\quad\quad & $\opname{Schur_{L}}C$ \\
\midrule
\endhead
\midrule
\multicolumn{3}{r}{Continued on next page}\\
\endfoot
\bottomrule
\endlastfoot
$\begin{tikzpicture}
      [baseline={([yshift=-.5ex]current bounding box.center)}, scale=0.4,  every node/.style={scale=0.5}]
      \node (1) at (0,0)  {};
      \node (21) at (1,1)  {};
      \node (2) at (2,0)[black] {};

      \draw[->] (1) -- (21);
      \draw[->] (21) -- (2);
    \end{tikzpicture}$ &
$\begin{tikzpicture}
      [baseline={([yshift=-.5ex]current bounding box.center)}, scale=0.4,  every node/.style={scale=0.5}]
      \node (1) at (0,0) {};
      \node (21) at (1,1)  {};
      \node (2) at (2,0) [black]{};
      \node (321) at (2,2) {};
      \node (32) at (3,1) {};
      \node (3) at (4,0) {};
      \node (4321) at (3,3) {};
      \node (432) at (4,2) {};
      \node (43) at (5,1) {};
      \node (4) at (6,0)[black] {};

      \node at (0,-.2) {};
      \node at (0,2.2) {};

      \draw[->] (1) -- (21);
      \draw[->] (21) -- (2);
      \draw[->] (21) -- (321);
      \draw[->] (321) -- (4321);
      \draw[->] (321) -- (32);
      \draw[->] (2) -- (32);
      \draw[->] (32) -- (3);
      \draw[->] (32) -- (432);
      \draw[->] (4321) -- (432);
      \draw[->] (432) -- (43);
      \draw[->] (3) -- (43);
      \draw[->] (43) -- (4);
    \end{tikzpicture}$ \quad\quad
   & \quad\quad
  $\begin{tikzpicture}
      [baseline={([yshift=-.5ex]current bounding box.center)}, scale=0.4,  every node/.style={scale=0.5}]
      \node (1) at (0,0)  {};
      \node (21) at (1,1)  {};
      \node (2) at (2,0)[black] {};

      \draw[->] (1) -- (21);
      \draw[->] (21) -- (2);
    \end{tikzpicture}$  \\
\midrule
$\begin{tikzpicture}
      [baseline={([yshift=-.5ex]current bounding box.center)}, scale=0.4,  every node/.style={scale=0.5}]
      \node (1) at (0,0)  {};
      \node (21) at (1,1)  {};
      \node (2) at (2,0)[black] {};

      \draw[->] (1) -- (21);
      \draw[->] (21) -- (2);
    \end{tikzpicture}$ &
$\begin{tikzpicture}
      [baseline={([yshift=-.5ex]current bounding box.center)}, scale=0.4,  every node/.style={scale=0.5}]
      \node (1) at (0,0) {};
      \node (21) at (1,1)  {};
      \node (2) at (2,0)[black] {};
      \node (321) at (2,2) {};
      \node (32) at (3,1) [white]{};
      \node (3) at (4,0)[black] {};
      \node (4321) at (3,3) {};
      \node (432) at (4,2) {};
      \node (43) at (5,1) {};
      \node (4) at (6,0) {};

      \node at (0,-.2) {};
      \node at (0,2.2) {};

      \draw[->] (1) -- (21);
      \draw[->] (21) -- (2);
      \draw[->] (21) -- (321);
      \draw[->] (321) -- (4321);
      \draw[->] (321) -- (32);
      \draw[->] (2) -- (32);
      \draw[->] (32) -- (3);
      \draw[->] (32) -- (432);
      \draw[->] (4321) -- (432);
      \draw[->] (432) -- (43);
      \draw[->] (3) -- (43);
      \draw[->] (43) -- (4);
    \end{tikzpicture}$ \quad\quad
   & \quad\quad
  $\begin{tikzpicture}
      [baseline={([yshift=-.5ex]current bounding box.center)}, scale=0.4,  every node/.style={scale=0.5}]
      \node (1) at (0,0)[black]  {};
      \node (21) at (1,1)  {};
      \node (2) at (2,0) {};

      \draw[->] (1) -- (21);
      \draw[->] (21) -- (2);
    \end{tikzpicture}$ \\
\midrule
$\begin{tikzpicture}
      [baseline={([yshift=-.5ex]current bounding box.center)}, scale=0.4,  every node/.style={scale=0.5}]
      \node (1) at (0,0)  {};
      \node (21) at (1,1)  {};
      \node (2) at (2,0)[black] {};

      \draw[->] (1) -- (21);
      \draw[->] (21) -- (2);
    \end{tikzpicture}$ &
\color{red}$\begin{tikzpicture}
      [baseline={([yshift=-.5ex]current bounding box.center)}, scale=0.4,  every node/.style={scale=0.5}]
      \node (1) at (0,0) {};
      \node (21) at (1,1)  {};
      \node (2) at (2,0)[black] {};
      \node (321) at (2,2) {};
      \node (32) at (3,1) {};
      \node (3) at (4,0) {};
      \node (4321) at (3,3) {};
      \node (432) at (4,2) [white]{};
      \node (43) at (5,1)[black] {};
      \node (4) at (6,0) {};

      \node at (0,-.2) {};
      \node at (0,2.2) {};

      \draw[->] (1) -- (21);
      \draw[->] (21) -- (2);
      \draw[->] (21) -- (321);
      \draw[->] (321) -- (4321);
      \draw[->] (321) -- (32);
      \draw[->] (2) -- (32);
      \draw[->] (32) -- (3);
      \draw[->] (32) -- (432);
      \draw[->] (4321) -- (432);
      \draw[->] (432) -- (43);
      \draw[->] (3) -- (43);
      \draw[->] (43) -- (4);
    \end{tikzpicture}$ \quad\quad
   & \quad\quad
\color{red} $\begin{tikzpicture}
      [baseline={([yshift=-.5ex]current bounding box.center)}, scale=0.4,  every node/.style={scale=0.5}]
      \node (1) at (0,0)  {};
      \node (21) at (1,1) [black] {};
      \node (2) at (2,0) {};

      \draw[->] (1) -- (21);
      \draw[->] (21) -- (2);
    \end{tikzpicture}$ \\
\midrule
$\begin{tikzpicture}
      [baseline={([yshift=-.5ex]current bounding box.center)}, scale=0.4,  every node/.style={scale=0.5}]
      \node (1) at (0,0)  {};
      \node (21) at (1,1)  {};
      \node (2) at (2,0)[black] {};

      \draw[->] (1) -- (21);
      \draw[->] (21) -- (2);
    \end{tikzpicture}$ &
\color{blue}$\begin{tikzpicture}
      [baseline={([yshift=-.5ex]current bounding box.center)}, scale=0.4,  every node/.style={scale=0.5}]
      \node (1) at (0,0) {};
      \node (21) at (1,1)  {};
      \node (2) at (2,0)[black] {};
      \node (321) at (2,2) {};
      \node (32) at (3,1) [white]{};
      \node (3) at (4,0)[black] {};
      \node (4321) at (3,3) {};
      \node (432) at (4,2)[white] {};
      \node (43) at (5,1) [black]{};
      \node (4) at (6,0) {};

      \node at (0,-.2) {};
      \node at (0,2.2) {};

      \draw[->] (1) -- (21);
      \draw[->] (21) -- (2);
      \draw[->] (21) -- (321);
      \draw[->] (321) -- (4321);
      \draw[->] (321) -- (32);
      \draw[->] (2) -- (32);
      \draw[->] (32) -- (3);
      \draw[->] (32) -- (432);
      \draw[->] (4321) -- (432);
      \draw[->] (432) -- (43);
      \draw[->] (3) -- (43);
      \draw[->] (43) -- (4);
    \end{tikzpicture}$ \quad\quad
   & \quad\quad
\color{blue}  $\begin{tikzpicture}
      [baseline={([yshift=-.5ex]current bounding box.center)}, scale=0.4,  every node/.style={scale=0.5}]
      \node (1) at (0,0) [black] {};
      \node (21) at (1,1)[black]  {};
      \node (2) at (2,0) {};

      \draw[->] (1) -- (21);
      \draw[->] (21) -- (2);
    \end{tikzpicture}$ \\
\midrule
$\begin{tikzpicture}
      [baseline={([yshift=-.5ex]current bounding box.center)}, scale=0.4,  every node/.style={scale=0.5}]
      \node (1) at (0,0) {};
      \node (21) at (1,1)  {};
      \node (2) at (2,0)[black] {};

      \draw[->] (1) -- (21);
      \draw[->] (21) -- (2);
    \end{tikzpicture}$ &
$\begin{tikzpicture}
      [baseline={([yshift=-.5ex]current bounding box.center)}, scale=0.4,  every node/.style={scale=0.5}]
      \node (1) at (0,0) {};
      \node (21) at (1,1)  {};
      \node (2) at (2,0)[black] {};
      \node (321) at (2,2) {};
      \node (32) at (3,1)[white] {};
      \node (3) at (4,0)[black] {};
      \node (4321) at (3,3) {};
      \node (432) at (4,2) [white]{};
      \node (43) at (5,1)[white] {};
      \node (4) at (6,0)[black] {};

      \node at (0,-.2) {};
      \node at (0,2.2) {};

      \draw[->] (1) -- (21);
      \draw[->] (21) -- (2);
      \draw[->] (21) -- (321);
      \draw[->] (321) -- (4321);
      \draw[->] (321) -- (32);
      \draw[->] (2) -- (32);
      \draw[->] (32) -- (3);
      \draw[->] (32) -- (432);
      \draw[->] (4321) -- (432);
      \draw[->] (432) -- (43);
      \draw[->] (3) -- (43);
      \draw[->] (43) -- (4);
    \end{tikzpicture}$ \quad\quad
   & \quad\quad
  $\begin{tikzpicture}
      [baseline={([yshift=-.5ex]current bounding box.center)}, scale=0.4,  every node/.style={scale=0.5}]
      \node (1) at (0,0) [black] {};
      \node (21) at (1,1) [white] {};
      \node (2) at (2,0)[black] {};

      \draw[->] (1) -- (21);
      \draw[->] (21) -- (2);
    \end{tikzpicture}$ \\
\midrule
$\begin{tikzpicture}
      [baseline={([yshift=-.5ex]current bounding box.center)}, scale=0.4,  every node/.style={scale=0.5}]
      \node (1) at (0,0) [black] {};
      \node (21) at (1,1)  {};
      \node (2) at (2,0) {};

      \draw[->] (1) -- (21);
      \draw[->] (21) -- (2);
    \end{tikzpicture}$ &
$\begin{tikzpicture}
      [baseline={([yshift=-.5ex]current bounding box.center)}, scale=0.4,  every node/.style={scale=0.5}]
      \node (1) at (0,0) [black]{};
      \node (21) at (1,1)  {};
      \node (2) at (2,0) {};
      \node (321) at (2,2) {};
      \node (32) at (3,1) {};
      \node (3) at (4,0) {};
      \node (4321) at (3,3) {};
      \node (432) at (4,2) {};
      \node (43) at (5,1) {};
      \node (4) at (6,0) [black]{};

      \node at (0,-.2) {};
      \node at (0,2.2) {};

      \draw[->] (1) -- (21);
      \draw[->] (21) -- (2);
      \draw[->] (21) -- (321);
      \draw[->] (321) -- (4321);
      \draw[->] (321) -- (32);
      \draw[->] (2) -- (32);
      \draw[->] (32) -- (3);
      \draw[->] (32) -- (432);
      \draw[->] (4321) -- (432);
      \draw[->] (432) -- (43);
      \draw[->] (3) -- (43);
      \draw[->] (43) -- (4);
    \end{tikzpicture}$ \quad\quad
   & \quad\quad
  $\begin{tikzpicture}
      [baseline={([yshift=-.5ex]current bounding box.center)}, scale=0.4,  every node/.style={scale=0.5}]
      \node (1) at (0,0)  {};
      \node (21) at (1,1)  {};
      \node (2) at (2,0)[black] {};

      \draw[->] (1) -- (21);
      \draw[->] (21) -- (2);
    \end{tikzpicture}$ \\
\midrule
$\begin{tikzpicture}
      [baseline={([yshift=-.5ex]current bounding box.center)}, scale=0.4,  every node/.style={scale=0.5}]
      \node (1) at (0,0) [black] {};
      \node (21) at (1,1) {};
      \node (2) at (2,0) {};

      \draw[->] (1) -- (21);
      \draw[->] (21) -- (2);
    \end{tikzpicture}$ &
$\begin{tikzpicture}
      [baseline={([yshift=-.5ex]current bounding box.center)}, scale=0.4,  every node/.style={scale=0.5}]
      \node (1) at (0,0) [black]{};
      \node (21) at (1,1)  {};
      \node (2) at (2,0) {};
      \node (321) at (2,2) {};
      \node (32) at (3,1) {};
      \node (3) at (4,0)[black] {};
      \node (4321) at (3,3) {};
      \node (432) at (4,2) {};
      \node (43) at (5,1) {};
      \node (4) at (6,0) {};

      \node at (0,-.2) {};
      \node at (0,2.2) {};

      \draw[->] (1) -- (21);
      \draw[->] (21) -- (2);
      \draw[->] (21) -- (321);
      \draw[->] (321) -- (4321);
      \draw[->] (321) -- (32);
      \draw[->] (2) -- (32);
      \draw[->] (32) -- (3);
      \draw[->] (32) -- (432);
      \draw[->] (4321) -- (432);
      \draw[->] (432) -- (43);
      \draw[->] (3) -- (43);
      \draw[->] (43) -- (4);
    \end{tikzpicture}$ \quad\quad
   & \quad\quad
$\begin{tikzpicture}
      [baseline={([yshift=-.5ex]current bounding box.center)}, scale=0.4,  every node/.style={scale=0.5}]
      \node (1) at (0,0) [black] {};
      \node (21) at (1,1)  {};
      \node (2) at (2,0) {};

      \draw[->] (1) -- (21);
      \draw[->] (21) -- (2);
    \end{tikzpicture}$ \\
\midrule
$\begin{tikzpicture}
      [baseline={([yshift=-.5ex]current bounding box.center)}, scale=0.4,  every node/.style={scale=0.5}]
      \node (1) at (0,0) [black] {};
      \node (21) at (1,1)  {};
      \node (2) at (2,0) {};

      \draw[->] (1) -- (21);
      \draw[->] (21) -- (2);
    \end{tikzpicture}$ &
\color{red}$\begin{tikzpicture}
      [baseline={([yshift=-.5ex]current bounding box.center)}, scale=0.4,  every node/.style={scale=0.5}]
      \node (1) at (0,0)[black] {};
      \node (21) at (1,1)  {};
      \node (2) at (2,0) {};
      \node (321) at (2,2) {};
      \node (32) at (3,1) {};
      \node (3) at (4,0) {};
      \node (4321) at (3,3) {};
      \node (432) at (4,2) {};
      \node (43) at (5,1)[black] {};
      \node (4) at (6,0) {};

      \node at (0,-.2) {};
      \node at (0,2.2) {};

      \draw[->] (1) -- (21);
      \draw[->] (21) -- (2);
      \draw[->] (21) -- (321);
      \draw[->] (321) -- (4321);
      \draw[->] (321) -- (32);
      \draw[->] (2) -- (32);
      \draw[->] (32) -- (3);
      \draw[->] (32) -- (432);
      \draw[->] (4321) -- (432);
      \draw[->] (432) -- (43);
      \draw[->] (3) -- (43);
      \draw[->] (43) -- (4);
    \end{tikzpicture}$ \quad\quad
   & \quad\quad
\color{red}  $\begin{tikzpicture}
      [baseline={([yshift=-.5ex]current bounding box.center)}, scale=0.4,  every node/.style={scale=0.5}]
      \node (1) at (0,0) {};
      \node (21) at (1,1) [black]  {};
      \node (2) at (2,0) {};

      \draw[->] (1) -- (21);
      \draw[->] (21) -- (2);
    \end{tikzpicture}$ \\
\midrule
$\begin{tikzpicture}
      [baseline={([yshift=-.5ex]current bounding box.center)}, scale=0.4,  every node/.style={scale=0.5}]
      \node (1) at (0,0) [black] {};
      \node (21) at (1,1) {};
      \node (2) at (2,0) {};

      \draw[->] (1) -- (21);
      \draw[->] (21) -- (2);
    \end{tikzpicture}$ &
\color{blue}$\begin{tikzpicture}
      [baseline={([yshift=-.5ex]current bounding box.center)}, scale=0.4,  every node/.style={scale=0.5}]
      \node (1) at (0,0) [black]{};
      \node (21) at (1,1)  {};
      \node (2) at (2,0) {};
      \node (321) at (2,2) {};
      \node (32) at (3,1) {};
      \node (3) at (4,0)[black] {};
      \node (4321) at (3,3) {};
      \node (432) at (4,2) {};
      \node (43) at (5,1)[black] {};
      \node (4) at (6,0) {};

      \node at (0,-.2) {};
      \node at (0,2.2) {};

      \draw[->] (1) -- (21);
      \draw[->] (21) -- (2);
      \draw[->] (21) -- (321);
      \draw[->] (321) -- (4321);
      \draw[->] (321) -- (32);
      \draw[->] (2) -- (32);
      \draw[->] (32) -- (3);
      \draw[->] (32) -- (432);
      \draw[->] (4321) -- (432);
      \draw[->] (432) -- (43);
      \draw[->] (3) -- (43);
      \draw[->] (43) -- (4);
    \end{tikzpicture}$ \quad\quad
   & \quad\quad
\color{blue}$\begin{tikzpicture}
      [baseline={([yshift=-.5ex]current bounding box.center)}, scale=0.4,  every node/.style={scale=0.5}]
      \node (1) at (0,0) [black] {};
      \node (21) at (1,1)[black]  {};
      \node (2) at (2,0) {};

      \draw[->] (1) -- (21);
      \draw[->] (21) -- (2);
    \end{tikzpicture}$ \\
\midrule
$\begin{tikzpicture}
      [baseline={([yshift=-.5ex]current bounding box.center)}, scale=0.4,  every node/.style={scale=0.5}]
      \node (1) at (0,0) [black] {};
      \node (21) at (1,1) {};
      \node (2) at (2,0) {};

      \draw[->] (1) -- (21);
      \draw[->] (21) -- (2);
    \end{tikzpicture}$ &
$\begin{tikzpicture}
      [baseline={([yshift=-.5ex]current bounding box.center)}, scale=0.4,  every node/.style={scale=0.5}]
      \node (1) at (0,0)[black] {};
      \node (21) at (1,1)  {};
      \node (2) at (2,0) {};
      \node (321) at (2,2) {};
      \node (32) at (3,1) {};
      \node (3) at (4,0)[black] {};
      \node (4321) at (3,3) {};
      \node (432) at (4,2) {};
      \node (43) at (5,1)[white] {};
      \node (4) at (6,0)[black] {};

      \node at (0,-.2) {};
      \node at (0,2.2) {};

      \draw[->] (1) -- (21);
      \draw[->] (21) -- (2);
      \draw[->] (21) -- (321);
      \draw[->] (321) -- (4321);
      \draw[->] (321) -- (32);
      \draw[->] (2) -- (32);
      \draw[->] (32) -- (3);
      \draw[->] (32) -- (432);
      \draw[->] (4321) -- (432);
      \draw[->] (432) -- (43);
      \draw[->] (3) -- (43);
      \draw[->] (43) -- (4);
    \end{tikzpicture}$ \quad\quad
   & \quad\quad
 $\begin{tikzpicture}
      [baseline={([yshift=-.5ex]current bounding box.center)}, scale=0.4,  every node/.style={scale=0.5}]
      \node (1) at (0,0) [black] {};
      \node (21) at (1,1) [white] {};
      \node (2) at (2,0)[black] {};

      \draw[->] (1) -- (21);
      \draw[->] (21) -- (2);
    \end{tikzpicture}$ \\
\midrule
\color{red}$\begin{tikzpicture}
      [baseline={([yshift=-.5ex]current bounding box.center)}, scale=0.4,  every node/.style={scale=0.5}]
      \node (1) at (0,0)  {};
      \node (21) at (1,1) [black] {};
      \node (2) at (2,0) {};

      \draw[->] (1) -- (21);
      \draw[->] (21) -- (2);
    \end{tikzpicture}$ &
\color{red}$\begin{tikzpicture}
      [baseline={([yshift=-.5ex]current bounding box.center)}, scale=0.4,  every node/.style={scale=0.5}]
      \node (1) at (0,0) {};
      \node (21) at (1,1) [black] {};
      \node (2) at (2,0) {};
      \node (321) at (2,2) {};
      \node (32) at (3,1) {};
      \node (3) at (4,0) {};
      \node (4321) at (3,3) {};
      \node (432) at (4,2) {};
      \node (43) at (5,1) {};
      \node (4) at (6,0)[black] {};

      \node at (0,-.2) {};
      \node at (0,2.2) {};

      \draw[->] (1) -- (21);
      \draw[->] (21) -- (2);
      \draw[->] (21) -- (321);
      \draw[->] (321) -- (4321);
      \draw[->] (321) -- (32);
      \draw[->] (2) -- (32);
      \draw[->] (32) -- (3);
      \draw[->] (32) -- (432);
      \draw[->] (4321) -- (432);
      \draw[->] (432) -- (43);
      \draw[->] (3) -- (43);
      \draw[->] (43) -- (4);
    \end{tikzpicture}$ \quad\quad
   & \quad\quad
 $\begin{tikzpicture}
      [baseline={([yshift=-.5ex]current bounding box.center)}, scale=0.4,  every node/.style={scale=0.5}]
      \node (1) at (0,0)  {};
      \node (21) at (1,1)  {};
      \node (2) at (2,0)[black] {};

      \draw[->] (1) -- (21);
      \draw[->] (21) -- (2);
    \end{tikzpicture}$ \\
\midrule
\color{red}$\begin{tikzpicture}
      [baseline={([yshift=-.5ex]current bounding box.center)}, scale=0.4,  every node/.style={scale=0.5}]
      \node (1) at (0,0)  {};
      \node (21) at (1,1)[black]  {};
      \node (2) at (2,0) {};

      \draw[->] (1) -- (21);
      \draw[->] (21) -- (2);
    \end{tikzpicture}$ &
\color{red}$\begin{tikzpicture}
      [baseline={([yshift=-.5ex]current bounding box.center)}, scale=0.4,  every node/.style={scale=0.5}]
      \node (1) at (0,0) {};
      \node (21) at (1,1) [black] {};
      \node (2) at (2,0) {};
      \node (321) at (2,2)[white] {};
      \node (32) at (3,1) {};
      \node (3) at (4,0)[black] {};
      \node (4321) at (3,3) {};
      \node (432) at (4,2) {};
      \node (43) at (5,1) {};
      \node (4) at (6,0) {};

      \node at (0,-.2) {};
      \node at (0,2.2) {};

      \draw[->] (1) -- (21);
      \draw[->] (21) -- (2);
      \draw[->] (21) -- (321);
      \draw[->] (321) -- (4321);
      \draw[->] (321) -- (32);
      \draw[->] (2) -- (32);
      \draw[->] (32) -- (3);
      \draw[->] (32) -- (432);
      \draw[->] (4321) -- (432);
      \draw[->] (432) -- (43);
      \draw[->] (3) -- (43);
      \draw[->] (43) -- (4);
    \end{tikzpicture}$ \quad\quad
   & \quad\quad
  $\begin{tikzpicture}
      [baseline={([yshift=-.5ex]current bounding box.center)}, scale=0.4,  every node/.style={scale=0.5}]
      \node (1) at (0,0) [black] {};
      \node (21) at (1,1)  {};
      \node (2) at (2,0) {};

      \draw[->] (1) -- (21);
      \draw[->] (21) -- (2);
    \end{tikzpicture}$  \\
\midrule
\color{red}$\begin{tikzpicture}
      [baseline={([yshift=-.5ex]current bounding box.center)}, scale=0.4,  every node/.style={scale=0.5}]
      \node (1) at (0,0)  {};
      \node (21) at (1,1) [black] {};
      \node (2) at (2,0) {};

      \draw[->] (1) -- (21);
      \draw[->] (21) -- (2);
    \end{tikzpicture}$ &
\color{red}$\begin{tikzpicture}
      [baseline={([yshift=-.5ex]current bounding box.center)}, scale=0.4,  every node/.style={scale=0.5}]
      \node (1) at (0,0) {};
      \node (21) at (1,1) [black] {};
      \node (2) at (2,0) {};
      \node (321) at (2,2) {};
      \node (32) at (3,1) {};
      \node (3) at (4,0) {};
      \node (4321) at (3,3)[white] {};
      \node (432) at (4,2) {};
      \node (43) at (5,1)[black] {};
      \node (4) at (6,0) {};

      \node at (0,-.2) {};
      \node at (0,2.2) {};

      \draw[->] (1) -- (21);
      \draw[->] (21) -- (2);
      \draw[->] (21) -- (321);
      \draw[->] (321) -- (4321);
      \draw[->] (321) -- (32);
      \draw[->] (2) -- (32);
      \draw[->] (32) -- (3);
      \draw[->] (32) -- (432);
      \draw[->] (4321) -- (432);
      \draw[->] (432) -- (43);
      \draw[->] (3) -- (43);
      \draw[->] (43) -- (4);
    \end{tikzpicture}$ \quad\quad
   & \quad\quad
\color{red}  $\begin{tikzpicture}
      [baseline={([yshift=-.5ex]current bounding box.center)}, scale=0.4,  every node/.style={scale=0.5}]
      \node (1) at (0,0)  {};
      \node (21) at (1,1)[black]  {};
      \node (2) at (2,0) {};

      \draw[->] (1) -- (21);
      \draw[->] (21) -- (2);
    \end{tikzpicture}$  \\
  \midrule
\color{red} $\begin{tikzpicture}
      [baseline={([yshift=-.5ex]current bounding box.center)}, scale=0.4,  every node/.style={scale=0.5}]
      \node (1) at (0,0)  {};
      \node (21) at (1,1) [black] {};
      \node (2) at (2,0) {};

      \draw[->] (1) -- (21);
      \draw[->] (21) -- (2);
    \end{tikzpicture}$ &
\color{yellow}$\begin{tikzpicture}
      [baseline={([yshift=-.5ex]current bounding box.center)}, scale=0.4,  every node/.style={scale=0.5}]
      \node (1) at (0,0) {};
      \node (21) at (1,1)[black]  {};
      \node (2) at (2,0) {};
      \node (321) at (2,2) [white]{};
      \node (32) at (3,1) {};
      \node (3) at (4,0)[black] {};
      \node (4321) at (3,3)[white] {};
      \node (432) at (4,2) {};
      \node (43) at (5,1)[black] {};
      \node (4) at (6,0) {};

      \node at (0,-.2) {};
      \node at (0,2.2) {};

      \draw[->] (1) -- (21);
      \draw[->] (21) -- (2);
      \draw[->] (21) -- (321);
      \draw[->] (321) -- (4321);
      \draw[->] (321) -- (32);
      \draw[->] (2) -- (32);
      \draw[->] (32) -- (3);
      \draw[->] (32) -- (432);
      \draw[->] (4321) -- (432);
      \draw[->] (432) -- (43);
      \draw[->] (3) -- (43);
      \draw[->] (43) -- (4);
    \end{tikzpicture}$ \quad\quad
   & \quad\quad
 \color{blue} $\begin{tikzpicture}
      [baseline={([yshift=-.5ex]current bounding box.center)}, scale=0.4,  every node/.style={scale=0.5}]
      \node (1) at (0,0)[black]  {};
      \node (21) at (1,1) [black] {};
      \node (2) at (2,0) {};

      \draw[->] (1) -- (21);
      \draw[->] (21) -- (2);
    \end{tikzpicture}$  \\
  \midrule
\color{red}  $\begin{tikzpicture}
      [baseline={([yshift=-.5ex]current bounding box.center)}, scale=0.4,  every node/.style={scale=0.5}]
      \node (1) at (0,0)  {};
      \node (21) at (1,1)[black]  {};
      \node (2) at (2,0) {};

      \draw[->] (1) -- (21);
      \draw[->] (21) -- (2);
    \end{tikzpicture}$ &
\color{red}$\begin{tikzpicture}
      [baseline={([yshift=-.5ex]current bounding box.center)}, scale=0.4,  every node/.style={scale=0.5}]
      \node (1) at (0,0) {};
      \node (21) at (1,1)[black]  {};
      \node (2) at (2,0) {};
      \node (321) at (2,2)[white] {};
      \node (32) at (3,1) {};
      \node (3) at (4,0) [black]{};
      \node (4321) at (3,3)[white] {};
      \node (432) at (4,2) {};
      \node (43) at (5,1) [white]{};
      \node (4) at (6,0)[black] {};

      \node at (0,-.2) {};
      \node at (0,2.2) {};

      \draw[->] (1) -- (21);
      \draw[->] (21) -- (2);
      \draw[->] (21) -- (321);
      \draw[->] (321) -- (4321);
      \draw[->] (321) -- (32);
      \draw[->] (2) -- (32);
      \draw[->] (32) -- (3);
      \draw[->] (32) -- (432);
      \draw[->] (4321) -- (432);
      \draw[->] (432) -- (43);
      \draw[->] (3) -- (43);
      \draw[->] (43) -- (4);
    \end{tikzpicture}$ \quad\quad
   & \quad\quad
  $\begin{tikzpicture}
      [baseline={([yshift=-.5ex]current bounding box.center)}, scale=0.4,  every node/.style={scale=0.5}]
      \node (1) at (0,0)[black]  {};
      \node (21) at (1,1) [white] {};
      \node (2) at (2,0)[black] {};

      \draw[->] (1) -- (21);
      \draw[->] (21) -- (2);
    \end{tikzpicture}$  \\
  \midrule
\color{blue}  $\begin{tikzpicture}
      [baseline={([yshift=-.5ex]current bounding box.center)}, scale=0.4,  every node/.style={scale=0.5}]
      \node (1) at (0,0) [black] {};
      \node (21) at (1,1)[black]  {};
      \node (2) at (2,0) {};

      \draw[->] (1) -- (21);
      \draw[->] (21) -- (2);
    \end{tikzpicture}$ &
\color{blue}$\begin{tikzpicture}
      [baseline={([yshift=-.5ex]current bounding box.center)}, scale=0.4,  every node/.style={scale=0.5}]
      \node (1) at (0,0) [black]{};
      \node (21) at (1,1) [black] {};
      \node (2) at (2,0) {};
      \node (321) at (2,2) {};
      \node (32) at (3,1) {};
      \node (3) at (4,0) {};
      \node (4321) at (3,3) {};
      \node (432) at (4,2) {};
      \node (43) at (5,1) {};
      \node (4) at (6,0)[black] {};

      \node at (0,-.2) {};
      \node at (0,2.2) {};

      \draw[->] (1) -- (21);
      \draw[->] (21) -- (2);
      \draw[->] (21) -- (321);
      \draw[->] (321) -- (4321);
      \draw[->] (321) -- (32);
      \draw[->] (2) -- (32);
      \draw[->] (32) -- (3);
      \draw[->] (32) -- (432);
      \draw[->] (4321) -- (432);
      \draw[->] (432) -- (43);
      \draw[->] (3) -- (43);
      \draw[->] (43) -- (4);
    \end{tikzpicture}$ \quad\quad
   & \quad\quad
  $\begin{tikzpicture}
      [baseline={([yshift=-.5ex]current bounding box.center)}, scale=0.4,  every node/.style={scale=0.5}]
      \node (1) at (0,0)  {};
      \node (21) at (1,1)  {};
      \node (2) at (2,0)[black] {};

      \draw[->] (1) -- (21);
      \draw[->] (21) -- (2);
    \end{tikzpicture}$  \\
  \midrule
\color{blue}  $\begin{tikzpicture}
      [baseline={([yshift=-.5ex]current bounding box.center)}, scale=0.4,  every node/.style={scale=0.5}]
      \node (1) at (0,0) [black] {};
      \node (21) at (1,1) [black] {};
      \node (2) at (2,0) {};

      \draw[->] (1) -- (21);
      \draw[->] (21) -- (2);
    \end{tikzpicture}$ &
\color{blue}$\begin{tikzpicture}
      [baseline={([yshift=-.5ex]current bounding box.center)}, scale=0.4,  every node/.style={scale=0.5}]
      \node (1) at (0,0)[black] {};
      \node (21) at (1,1) [black] {};
      \node (2) at (2,0) {};
      \node (321) at (2,2)[white] {};
      \node (32) at (3,1) {};
      \node (3) at (4,0)[black] {};
      \node (4321) at (3,3) {};
      \node (432) at (4,2) {};
      \node (43) at (5,1) {};
      \node (4) at (6,0) {};

      \node at (0,-.2) {};
      \node at (0,2.2) {};

      \draw[->] (1) -- (21);
      \draw[->] (21) -- (2);
      \draw[->] (21) -- (321);
      \draw[->] (321) -- (4321);
      \draw[->] (321) -- (32);
      \draw[->] (2) -- (32);
      \draw[->] (32) -- (3);
      \draw[->] (32) -- (432);
      \draw[->] (4321) -- (432);
      \draw[->] (432) -- (43);
      \draw[->] (3) -- (43);
      \draw[->] (43) -- (4);
    \end{tikzpicture}$ \quad\quad
   & \quad\quad
  $\begin{tikzpicture}
      [baseline={([yshift=-.5ex]current bounding box.center)}, scale=0.4,  every node/.style={scale=0.5}]
      \node (1) at (0,0) [black] {};
      \node (21) at (1,1)  {};
      \node (2) at (2,0) {};

      \draw[->] (1) -- (21);
      \draw[->] (21) -- (2);
    \end{tikzpicture}$  \\
\midrule
\color{blue} $\begin{tikzpicture}
      [baseline={([yshift=-.5ex]current bounding box.center)}, scale=0.4,  every node/.style={scale=0.5}]
      \node (1) at (0,0)[black]  {};
      \node (21) at (1,1) [black] {};
      \node (2) at (2,0) {};

      \draw[->] (1) -- (21);
      \draw[->] (21) -- (2);
    \end{tikzpicture}$ &
\color{yellow}$\begin{tikzpicture}
      [baseline={([yshift=-.5ex]current bounding box.center)}, scale=0.4,  every node/.style={scale=0.5}]
      \node (1) at (0,0)[black] {};
      \node (21) at (1,1) [black] {};
      \node (2) at (2,0) {};
      \node (321) at (2,2) {};
      \node (32) at (3,1) {};
      \node (3) at (4,0) {};
      \node (4321) at (3,3)[white] {};
      \node (432) at (4,2) {};
      \node (43) at (5,1)[black] {};
      \node (4) at (6,0) {};

      \node at (0,-.2) {};
      \node at (0,2.2) {};

      \draw[->] (1) -- (21);
      \draw[->] (21) -- (2);
      \draw[->] (21) -- (321);
      \draw[->] (321) -- (4321);
      \draw[->] (321) -- (32);
      \draw[->] (2) -- (32);
      \draw[->] (32) -- (3);
      \draw[->] (32) -- (432);
      \draw[->] (4321) -- (432);
      \draw[->] (432) -- (43);
      \draw[->] (3) -- (43);
      \draw[->] (43) -- (4);
    \end{tikzpicture}$ \quad\quad
   & \quad\quad
 \color{red} $\begin{tikzpicture}
      [baseline={([yshift=-.5ex]current bounding box.center)}, scale=0.4,  every node/.style={scale=0.5}]
      \node (1) at (0,0)  {};
      \node (21) at (1,1) [black] {};
      \node (2) at (2,0) {};

      \draw[->] (1) -- (21);
      \draw[->] (21) -- (2);
    \end{tikzpicture}$  \\
\midrule
\color{blue} $\begin{tikzpicture}
      [baseline={([yshift=-.5ex]current bounding box.center)}, scale=0.4,  every node/.style={scale=0.5}]
      \node (1) at (0,0) [black] {};
      \node (21) at (1,1) [black] {};
      \node (2) at (2,0) {};

      \draw[->] (1) -- (21);
      \draw[->] (21) -- (2);
    \end{tikzpicture}$ &
\color{blue}$\begin{tikzpicture}
      [baseline={([yshift=-.5ex]current bounding box.center)}, scale=0.4,  every node/.style={scale=0.5}]
      \node (1) at (0,0) [black]{};
      \node (21) at (1,1) [black] {};
      \node (2) at (2,0) {};
      \node (321) at (2,2) [white]{};
      \node (32) at (3,1) {};
      \node (3) at (4,0)[black] {};
      \node (4321) at (3,3)[white] {};
      \node (432) at (4,2) {};
      \node (43) at (5,1)[black] {};
      \node (4) at (6,0) {};

      \node at (0,-.2) {};
      \node at (0,2.2) {};

      \draw[->] (1) -- (21);
      \draw[->] (21) -- (2);
      \draw[->] (21) -- (321);
      \draw[->] (321) -- (4321);
      \draw[->] (321) -- (32);
      \draw[->] (2) -- (32);
      \draw[->] (32) -- (3);
      \draw[->] (32) -- (432);
      \draw[->] (4321) -- (432);
      \draw[->] (432) -- (43);
      \draw[->] (3) -- (43);
      \draw[->] (43) -- (4);
    \end{tikzpicture}$ \quad\quad
   & \quad\quad
\color{blue}  $\begin{tikzpicture}
      [baseline={([yshift=-.5ex]current bounding box.center)}, scale=0.4,  every node/.style={scale=0.5}]
      \node (1) at (0,0) [black] {};
      \node (21) at (1,1) [black] {};
      \node (2) at (2,0) {};

      \draw[->] (1) -- (21);
      \draw[->] (21) -- (2);
    \end{tikzpicture}$  \\
\midrule
\color{blue} $\begin{tikzpicture}
      [baseline={([yshift=-.5ex]current bounding box.center)}, scale=0.4,  every node/.style={scale=0.5}]
      \node (1) at (0,0) [black] {};
      \node (21) at (1,1) [black] {};
      \node (2) at (2,0) {};

      \draw[->] (1) -- (21);
      \draw[->] (21) -- (2);
    \end{tikzpicture}$ &
\color{blue}$\begin{tikzpicture}
      [baseline={([yshift=-.5ex]current bounding box.center)}, scale=0.4,  every node/.style={scale=0.5}]
      \node (1) at (0,0)[black] {};
      \node (21) at (1,1)[black]  {};
      \node (2) at (2,0) {};
      \node (321) at (2,2)[white] {};
      \node (32) at (3,1) {};
      \node (3) at (4,0)[black] {};
      \node (4321) at (3,3) [white]{};
      \node (432) at (4,2) {};
      \node (43) at (5,1)[white] {};
      \node (4) at (6,0) [black]{};

      \node at (0,-.2) {};
      \node at (0,2.2) {};

      \draw[->] (1) -- (21);
      \draw[->] (21) -- (2);
      \draw[->] (21) -- (321);
      \draw[->] (321) -- (4321);
      \draw[->] (321) -- (32);
      \draw[->] (2) -- (32);
      \draw[->] (32) -- (3);
      \draw[->] (32) -- (432);
      \draw[->] (4321) -- (432);
      \draw[->] (432) -- (43);
      \draw[->] (3) -- (43);
      \draw[->] (43) -- (4);
    \end{tikzpicture}$ \quad\quad
   & \quad\quad
  $\begin{tikzpicture}
      [baseline={([yshift=-.5ex]current bounding box.center)}, scale=0.4,  every node/.style={scale=0.5}]
      \node (1) at (0,0)[black]  {};
      \node (21) at (1,1)[white]  {};
      \node (2) at (2,0) [black]{};

      \draw[->] (1) -- (21);
      \draw[->] (21) -- (2);
    \end{tikzpicture}$  \\
\midrule
 $\begin{tikzpicture}
      [baseline={([yshift=-.5ex]current bounding box.center)}, scale=0.4,  every node/.style={scale=0.5}]
      \node (1) at (0,0) [black] {};
      \node (21) at (1,1)[white]  {};
      \node (2) at (2,0) [black]{};

      \draw[->] (1) -- (21);
      \draw[->] (21) -- (2);
    \end{tikzpicture}$ &
$\begin{tikzpicture}
      [baseline={([yshift=-.5ex]current bounding box.center)}, scale=0.4,  every node/.style={scale=0.5}]
      \node (1) at (0,0)[black] {};
      \node (21) at (1,1) [white] {};
      \node (2) at (2,0)[black] {};
      \node (321) at (2,2) {};
      \node (32) at (3,1) {};
      \node (3) at (4,0) {};
      \node (4321) at (3,3) {};
      \node (432) at (4,2) {};
      \node (43) at (5,1) {};
      \node (4) at (6,0) [black]{};

      \node at (0,-.2) {};
      \node at (0,2.2) {};

      \draw[->] (1) -- (21);
      \draw[->] (21) -- (2);
      \draw[->] (21) -- (321);
      \draw[->] (321) -- (4321);
      \draw[->] (321) -- (32);
      \draw[->] (2) -- (32);
      \draw[->] (32) -- (3);
      \draw[->] (32) -- (432);
      \draw[->] (4321) -- (432);
      \draw[->] (432) -- (43);
      \draw[->] (3) -- (43);
      \draw[->] (43) -- (4);
    \end{tikzpicture}$ \quad\quad
   & \quad\quad
  $\begin{tikzpicture}
      [baseline={([yshift=-.5ex]current bounding box.center)}, scale=0.4,  every node/.style={scale=0.5}]
      \node (1) at (0,0)  {};
      \node (21) at (1,1)  {};
      \node (2) at (2,0)[black] {};

      \draw[->] (1) -- (21);
      \draw[->] (21) -- (2);
    \end{tikzpicture}$  \\
\midrule
 $\begin{tikzpicture}
      [baseline={([yshift=-.5ex]current bounding box.center)}, scale=0.4,  every node/.style={scale=0.5}]
      \node (1) at (0,0) [black] {};
      \node (21) at (1,1)[white]  {};
      \node (2) at (2,0)[black] {};

      \draw[->] (1) -- (21);
      \draw[->] (21) -- (2);
    \end{tikzpicture}$ &
$\begin{tikzpicture}
      [baseline={([yshift=-.5ex]current bounding box.center)}, scale=0.4,  every node/.style={scale=0.5}]
      \node (1) at (0,0)[black] {};
      \node (21) at (1,1)[white]  {};
      \node (2) at (2,0)[black] {};
      \node (321) at (2,2)[white] {};
      \node (32) at (3,1) {};
      \node (3) at (4,0)[black] {};
      \node (4321) at (3,3) {};
      \node (432) at (4,2) {};
      \node (43) at (5,1) {};
      \node (4) at (6,0) {};

      \node at (0,-.2) {};
      \node at (0,2.2) {};

      \draw[->] (1) -- (21);
      \draw[->] (21) -- (2);
      \draw[->] (21) -- (321);
      \draw[->] (321) -- (4321);
      \draw[->] (321) -- (32);
      \draw[->] (2) -- (32);
      \draw[->] (32) -- (3);
      \draw[->] (32) -- (432);
      \draw[->] (4321) -- (432);
      \draw[->] (432) -- (43);
      \draw[->] (3) -- (43);
      \draw[->] (43) -- (4);
    \end{tikzpicture}$ \quad\quad
   & \quad\quad
  $\begin{tikzpicture}
      [baseline={([yshift=-.5ex]current bounding box.center)}, scale=0.4,  every node/.style={scale=0.5}]
      \node (1) at (0,0) [black] {};
      \node (21) at (1,1)  {};
      \node (2) at (2,0) {};

      \draw[->] (1) -- (21);
      \draw[->] (21) -- (2);
    \end{tikzpicture}$  \\
\midrule
 $\begin{tikzpicture}
      [baseline={([yshift=-.5ex]current bounding box.center)}, scale=0.4,  every node/.style={scale=0.5}]
      \node (1) at (0,0)[black]  {};
      \node (21) at (1,1) [white] {};
      \node (2) at (2,0)[black] {};

      \draw[->] (1) -- (21);
      \draw[->] (21) -- (2);
    \end{tikzpicture}$ &
\color{red}$\begin{tikzpicture}
      [baseline={([yshift=-.5ex]current bounding box.center)}, scale=0.4,  every node/.style={scale=0.5}]
      \node (1) at (0,0)[black] {};
      \node (21) at (1,1) [white] {};
      \node (2) at (2,0)[black] {};
      \node (321) at (2,2) {};
      \node (32) at (3,1) {};
      \node (3) at (4,0) {};
      \node (4321) at (3,3)[white] {};
      \node (432) at (4,2) [white]{};
      \node (43) at (5,1)[black] {};
      \node (4) at (6,0) {};

      \node at (0,-.2) {};
      \node at (0,2.2) {};

      \draw[->] (1) -- (21);
      \draw[->] (21) -- (2);
      \draw[->] (21) -- (321);
      \draw[->] (321) -- (4321);
      \draw[->] (321) -- (32);
      \draw[->] (2) -- (32);
      \draw[->] (32) -- (3);
      \draw[->] (32) -- (432);
      \draw[->] (4321) -- (432);
      \draw[->] (432) -- (43);
      \draw[->] (3) -- (43);
      \draw[->] (43) -- (4);
    \end{tikzpicture}$ \quad\quad
   & \quad\quad
\color{red}  $\begin{tikzpicture}
      [baseline={([yshift=-.5ex]current bounding box.center)}, scale=0.4,  every node/.style={scale=0.5}]
      \node (1) at (0,0)  {};
      \node (21) at (1,1)[black]  {};
      \node (2) at (2,0) {};

      \draw[->] (1) -- (21);
      \draw[->] (21) -- (2);
    \end{tikzpicture}$  \\
\midrule
 $\begin{tikzpicture}
      [baseline={([yshift=-.5ex]current bounding box.center)}, scale=0.4,  every node/.style={scale=0.5}]
      \node (1) at (0,0) [black] {};
      \node (21) at (1,1)[white]  {};
      \node (2) at (2,0)[black] {};

      \draw[->] (1) -- (21);
      \draw[->] (21) -- (2);
    \end{tikzpicture}$ &
\color{blue}$\begin{tikzpicture}
      [baseline={([yshift=-.5ex]current bounding box.center)}, scale=0.4,  every node/.style={scale=0.5}]
      \node (1) at (0,0)[black] {};
      \node (21) at (1,1) [white] {};
      \node (2) at (2,0)[black] {};
      \node (321) at (2,2)[white] {};
      \node (32) at (3,1)[white] {};
      \node (3) at (4,0)[black] {};
      \node (4321) at (3,3)[white] {};
      \node (432) at (4,2)[white] {};
      \node (43) at (5,1)[black] {};
      \node (4) at (6,0) {};

      \node at (0,-.2) {};
      \node at (0,2.2) {};

      \draw[->] (1) -- (21);
      \draw[->] (21) -- (2);
      \draw[->] (21) -- (321);
      \draw[->] (321) -- (4321);
      \draw[->] (321) -- (32);
      \draw[->] (2) -- (32);
      \draw[->] (32) -- (3);
      \draw[->] (32) -- (432);
      \draw[->] (4321) -- (432);
      \draw[->] (432) -- (43);
      \draw[->] (3) -- (43);
      \draw[->] (43) -- (4);
    \end{tikzpicture}$ \quad\quad
   & \quad\quad
\color{blue}  $\begin{tikzpicture}
      [baseline={([yshift=-.5ex]current bounding box.center)}, scale=0.4,  every node/.style={scale=0.5}]
      \node (1) at (0,0) [black] {};
      \node (21) at (1,1)[black]  {};
      \node (2) at (2,0) {};

      \draw[->] (1) -- (21);
      \draw[->] (21) -- (2);
    \end{tikzpicture}$  \\

\end{longtable}
\endgroup
\end{example}

\section*{Acknowledgements}
This work was supported by the NSFC (Grant No. 12201211). The authors would like to thank the anonymous referee for the constructive and detailed comments, which have greatly improved the quality of this manuscript.

\bibliographystyle{abbrv}
%\bibliography{BCAmutation.bib}

\end{document}